\documentclass[a4paper,11pt,reqno]{amsart}

\usepackage{url}
\usepackage[colorlinks=true, linkcolor=blue, citecolor=ForestGreen]{hyperref}
\usepackage{cite}

\usepackage{latexsym,amsmath,amssymb,amsfonts,amsthm}
\usepackage{mathrsfs}
\usepackage{mathtools}
\usepackage{dsfont}
\usepackage{bbm}
\usepackage{soul}     
\usepackage[utf8]{inputenc}

\usepackage[usenames, dvipsnames]{color}
\usepackage[usenames, dvipsnames, cmyk]{xcolor}
\usepackage{graphicx}
\usepackage[scriptsize,hang,raggedright]{subfigure}
\usepackage{multirow}
\usepackage{enumerate}
\usepackage{enumitem}

\usepackage{a4wide}
\usepackage{epsfig,epstopdf}

\numberwithin{equation}{section}

\definecolor{grey}{rgb}{.7,.7,.7}
\definecolor{refkey}{gray}{.45}
\definecolor{labelkey}{gray}{.45}

\newcommand{\xupref}[2]{\hspace{-0.3ex}\stackrel{\eqref{#1}}{#2}} 

\newtheorem{theorem}{Theorem}[section]
\newtheorem{proposition}[theorem]{Proposition}
\newtheorem{lemma}[theorem]{Lemma}

\theoremstyle{remark}
\newtheorem{remark}[theorem]{Remark}
\theoremstyle{definition}
\newtheorem{definition}[theorem]{Definition}
\newtheorem{example}[theorem]{Example}

\usepackage{tikz}
\usepackage{pgf,pgfplots}
\usetikzlibrary{calc}
\usetikzlibrary{decorations.markings}
\usetikzlibrary{decorations.pathmorphing}
\usetikzlibrary{decorations.shapes}
\usetikzlibrary{shapes,arrows,shapes.geometric,patterns,fadings}

\newcommand{\R}{\mathbb R}

\newcommand{\dd}{\,\mathrm{d}}
\renewcommand{\setminus}{\backslash}

\newcommand{\defeq}{\coloneqq}

\newcommand{\nl}{\mathcal{E}}
\newcommand{\Hone}{\mathcal{H}^1}

\newcommand{\pol}{\mathscr{P}}
\newcommand{\p}{\mathcal{P}}
\newcommand{\side}[2]{\mkern2mu \overline{\mkern-2mu #1 #2}}
\newcommand{\ka}{\kappa}
\newcommand{\per}{\mathrm{Per}}

\newcommand{\calI}{\mathcal{I}}

\newcommand{\ba}{\begin{array}}
\newcommand{\ea}{\end{array}}

\newcommand{\tld}[1]{\widetilde{#1}}
\newcommand{\bthm}{\begin{theorem}}
\newcommand{\ethm}{\end{theorem}}
\newcommand{\bprop}{\begin{proposition}}
\newcommand{\eprop}{\end{proposition}}
\newcommand{\blemma}{\begin{lemma}}
\newcommand{\elemma}{\end{lemma}}
\newcommand{\bexmpl}{\begin{example}}
\newcommand{\eexmpl}{\end{example}}

\newcommand{\beqn}{\begin{equation}}
\newcommand{\eeqn}{\end{equation}}
\newcommand{\beqns}{\begin{equation*}}
\newcommand{\eeqns}{\end{equation*}}

\newcommand{\pr}{\prime}

\newcommand{\Rd}{\mathbb{R}^d}

\renewcommand{\leq}{\leqslant}
\renewcommand{\geq}{\geqslant}

\definecolor{mygreen}{rgb}{0.1,0.75,0.2}

\newcounter{myenumi}
\DeclareMathOperator{\dive}{div}

\DeclareMathOperator*{\dist}{dist}

\title[Overdetermined Problems for Triangles and Quadrilaterals]{Riesz-type Inequalities and Overdetermined Problems for Triangles and Quadrilaterals}

\author{Marco Bonacini}
\address[Marco Bonacini]{Department of Mathematics, University of Trento, Italy}
\email{marco.bonacini@unitn.it}

\author{Riccardo Cristoferi}
\address[Riccardo Cristoferi]{Department of Mathematics - IMAPP, Radboud University, Nijmegen, The Netherlands}
\email{riccardo.cristoferi@ru.nl}

\author{Ihsan Topaloglu}
\address[Ihsan Topaloglu]{Department of Mathematics and Applied Mathematics, Virginia Commonwealth University, Richmond, VA, USA}
\email{iatopaloglu@vcu.edu}

\date{\today}                                        

\thanks{This is a post-peer-review, pre-copyedit version of an article published in the Journal of Geometric Analysis. The final
authenticated version is available online at: \url{https://doi.org/10.1007/s12220-021-00737-7}.}                                                  

\begin{document}

\begin{abstract}
We consider Riesz-type nonlocal interaction energies over polygons. We prove the analog of the Riesz inequality in this discrete setting for triangles and quadrilaterals, and obtain that among all $N$-gons with fixed area, the nonlocal energy is maximized by a regular polygon, for $N=3,4$. Further we derive necessary first-order stationarity conditions for a polygon with respect to a restricted class of variations, which will then be used to characterize regular $N$-gons, for $N=3,4$, as solutions to an overdetermined free boundary problem.
\end{abstract}

\maketitle


\section{Introduction}\label{sec:intro}

In this paper we study a class of nonlocal repulsive energies of generalized Riesz-type on polygons. We consider the nonlocal energy
\begin{equation}\label{eq:nonlocal_energy}
\nl(E)\defeq \int_E\int_E K(|x-y|) \dd x \dd y
\end{equation}
defined on measurable subsets $E\subset\R^2$ with finite Lebesgue measure. We assume that the kernel $K$ satisfies the following assumptions:
\begin{enumerate}[label = (K\arabic*)]
	\addtolength{\itemsep}{6pt}
\item\label{ass:regular} $K\in C^1((0,\infty))$, $K\geq0$;
\item\label{ass:decreasing} $K$ is strictly decreasing;
\item\label{ass:integrable} $K$ satisfies
\begin{equation} \label{eq:assintegrable}
\int_0^1 K(r)\,r\dd r <\infty.
\end{equation}
\end{enumerate}
The kernel $K$ is possibly singular at the origin, and the integrability condition \eqref{eq:assintegrable} guarantees that the energy \eqref{eq:nonlocal_energy} is finite on sets with finite measure (see Remark~\ref{rmk:finiteness}). The prototype case is the Riesz kernel $K(r)=r^{-\alpha}$, with $\alpha\in(0,2)$.

It is well-known that the energy \eqref{eq:nonlocal_energy} (in any dimension) is uniquely maximized by the ball under volume constraint, as a consequence of Riesz's rearrangement inequality. Moreover, at least in the case of the Riesz kernels, balls are characterized as the unique critical points for the energy \eqref{eq:nonlocal_energy} under volume constraint, in the following sense. We define the \emph{potential} associated to a measurable set $E\subset\R^2$ with finite measure as
\begin{equation}\label{eq:potential}
v_E(x)\defeq \int_E K(|x-y|) \dd y,
\end{equation}
and say that a set $E$ is stationary for $\nl$ with respect to area-preserving variations if $v_E$ is constant on $\partial E$. It was proved in a series of contributions \cite{ChNeuTo20,Fraenkel00,LuZhu12,Reichel09} via moving plane methods, and in full generality for Riesz kernels in \cite{GomezSerrano_et_al_19} via a continuous Steiner symmerization argument, that balls are the only sets which enjoy this property: in other words, when defined over \emph{all} measurable sets of fixed measure, the overdetermined problem for the potential enforces the symmetry of the set.

The scope of this paper is to investigate the same two questions in a discrete setting, namely when restricting the class of sets on which we evaluate the energy to convex polygons with a fixed number of sides. While on the one hand this restriction simplifies some aspects of the problem by essentially reducing it to a finite dimensional problem, on the other hand it introduces new challenges and requires new techniques, as classical arguments such as moving plane methods do not apply in restricted classes.

\medskip

We first consider the problem of the area-constrained maximization of the nonlocal energy $\nl$ in the class $\pol_N$ of \emph{all} polygons in $\R^2$ with $N\geq3$ sides: for $m>0$, 
\begin{equation} \label{eq:max}
\max\bigl\{ \nl(\p) \,:\, \p\in\pol_N,\, |\p|=m \bigr\},
\end{equation}
where $|\mathcal{P}|\defeq\mathcal{L}^2(\mathcal{P})$ denotes the area of a polygon $\mathcal{P}\in\pol_N$. It is in general expected that for each fixed number of sides the regular $N$-gon is the unique maximizer of \eqref{eq:max}. In our first main result we show that this is true in the case of triangles and quadrilaterals.

\begin{theorem}\label{thm1}
The equilateral triangle is the unique (up to rigid movements) maximizer of $\nl$ in $\pol_3$ under area constraint, and the square is the unique (up to rigid movements) maximizer of $\nl$ in $\pol_4$ under area constraint.
\end{theorem}

The proof relies on the combination of two properties that had already been established in the literature and are well-known to experts: (a) the fact that the nonlocal energy is increasing under Steiner symmetrization of a set due to classical rearrangement inequalities (see \cite{Li2}); and, (b) the observation, originally due to P\'{o}lya and Szeg\H{o}, that for any given triangle or quadrilateral it is possible to find a sequence of Steiner symmetrizations which converge to an equilateral triangle or to a square, respectively. This strategy was used by P\'{o}lya and Szeg\H{o} \cite[p.~158]{PoSz51} to prove their conjecture about the optimality of the regular $N$-gon for various classical shape functionals, such as the principal eigenvalue of the Laplacian, the torsional rigidity, and the electrostatic capacity, for $N=3,4$. The main drawback of this approach is that, for more than four sides, it seems not possible to construct in an easy way a sequence of symmetrizations converging to the regular $N$-gon and preserving the number of sides at each step. Therefore the extension of Theorem~\ref{thm1} to the case $N\geq5$ seems to be, as far as we know, an interesting open problem.

Besides the above mentioned conjecture by P\'{o}lya and Szeg\H{o}, solved only for the logarithmic capacity in \cite{SoZa04}, the problem of optimality of regular $N$-gons for variational functionals has been the object of several contributions. Among these, we mention the papers \cite{BucFra21,FrGaLa13,Nit14}, dealing with various shape optimization problems on polygons involving spectral functionals, and \cite{BuFr16}, where it is proved that the regular polygon minimizes the Cheeger constant among polygons with fixed area and number of sides.

\medskip

Next, we turn to the second main question that we address in this paper, namely whether the regular $N$-gon is characterized by the stationarity conditions for problem \eqref{eq:max}, as it is the case for the ball. Of course, we need to consider a notion of criticality with respect to variations that preserve the polygonal structure and the number of sides. Following \cite{BuFr16,FraVel19}, in Section~\ref{sec:overdetermined} we introduce two specific classes of perturbations of a given polygon: the first is obtained by translating a side of the polygon parallel to itself, the second by rotating a side with respect to its midpoint. We then show that, for $N=3$ and $N=4$ sides, the unique $N$-gon which is stationary with respect to these two families of perturbations, under an area or a perimeter constraint, is the equilateral triangle or the square, respectively.

In order to state precisely our second main result, we need to fix some notation that will be used throughout the paper. Given two points $P,Q\in\R^2$, we denote by $\side{P}{Q}\defeq\{tP+(1-t)Q \,:\, t\in[0,1]\}$ the segment joining $P$ and $Q$. For $N\geq3$, let $\p\in\pol_N$ be a polygon with $N$ vertices $P_1,\ldots,P_N$. For notational convenience we also set $P_0\defeq P_N$, $P_{N+1}\defeq P_1$. We let for $i\in\{1,\ldots,N\}$:
\begin{itemize}
\item $\nu_i$ be the exterior unit normal to the side $\side{P_i}{P_{i+1}}$,
\item $\ell_i$ be the length of the side $\side{P_i}{P_{i+1}}$,
\item $\theta_i$ be the (interior) angle at the vertex $P_i$,
\item $M_i$ be the midpoint of the side $\side{P_i}{P_{i+1}}$.
\end{itemize}
Denoting by $v_\p$ the potential associated with the polygon $\p$ according to \eqref{eq:potential}, we then consider the following two conditions:
\begin{equation}\label{eq:sliding}
\frac{1}{\ell_i}\int_{\side{P_i}{P_{i+1}}} v_{\p}(x)\dd\Hone(x) = \frac{1}{\ell_j}\int_{\side{P_j}{P_{j+1}}} v_{\p}(x)\dd\Hone(x) \quad\text{for all }i,j\in\{1,\ldots,N\},
\end{equation}
which corresponds to the criticality condition for the energy $\nl$ under an area constraint, when sides are translated parallel to themselves, and
\begin{equation}\label{eq:tilting}
\int_{\side{P_i}{M_i}} v_{\p}(x)|x-M_i|\dd\Hone(x) = \int_{\side{P_{i+1}}{M_i}} v_{\p}(x)|x-M_i|\dd\Hone(x) \quad\text{for all }i\in\{1,\ldots,N\},
\end{equation}
which corresponds to the criticality condition for the energy $\nl$ under an area constraint, when a side is rotated around its midpoint. The derivation of \eqref{eq:sliding} and \eqref{eq:tilting} will be given in Section~\ref{sec:overdetermined}, see in particular Theorem~\ref{thm:stationarity}. Our second result is the following.

\begin{theorem}\label{thm2}
If $\p\in\pol_3$ obeys condition \eqref{eq:tilting}, then $\p$ is an equilateral triangle. If $\p\in\pol_4$ obeys conditions \eqref{eq:sliding} and \eqref{eq:tilting}, then $\p$ is a square.
\end{theorem}

We also prove the analogous of Theorem~\ref{thm2} when we replace \eqref{eq:sliding} and \eqref{eq:tilting} by the corresponding stationarity conditions, with respect to the same two families of perturbations, under a \emph{perimeter} constraint, namely
\begin{equation}\label{eq:sliding2}
\int_{\side{P_i}{P_{i+1}}} v_{\p}(x)\dd\Hone(x) = \bar{\sigma} \bigl(\psi(\theta_i)+\psi(\theta_{i+1})\bigr)
\end{equation}
and
\begin{equation}\label{eq:tilting2}
\int_{\side{P_i}{M_i}} v_{\p}(x)|x-M_i|\dd\Hone(x) - \int_{\side{P_{i+1}}{M_i}} v_{\p}(x)|x-M_i|\dd\Hone(x) = \frac{\bar{\sigma}\ell_i}{2}\bigl(\psi(\theta_i)-\psi(\theta_{i+1})\bigr)
\end{equation}
where $\bar{\sigma}$ is a positive constant (independent of $i$), and
\begin{equation} \label{eq:cotangent}
\psi(\theta)\defeq \cot\theta+\frac{1}{\sin\theta}
\end{equation}
(see again Section~\ref{sec:overdetermined} for the derivation). We then have the following.
\begin{theorem}\label{thm3}
If $\p\in\pol_3$ obeys condition \eqref{eq:tilting2}, then $\p$ is an equilateral triangle. If $\p\in\pol_4$ obeys conditions \eqref{eq:sliding2} and \eqref{eq:tilting2}, then $\p$ is a square.
\end{theorem}

These results can be interpreted as Serrin-type theorems yielding the characterization of the regular $N$-gon as the unique solution of the overdetermined problems \eqref{eq:potential}--\eqref{eq:sliding}--\eqref{eq:tilting}, or \eqref{eq:potential}--\eqref{eq:sliding2}--\eqref{eq:tilting2}, for the potential $v_\p$. Despite the large literature on overdetermined boundary value problems, symmetry results of this kind in a polygonal setting seem to have been considered only recently, with a first contribution by Fragalà and Velichkov \cite{FraVel19} which was also inspirational for our work. In \cite{FraVel19} it was proved that the overdetermined problem corresponding to the stationarity conditions for the torsional rigidity and for the first Dirichlet eigenvalue of the Laplacian, under an area or a perimeter constraint, characterizes the equilateral triangle among all triangles. We also mention the recent paper \cite{Sak21} for a related result, where equilateral triangles are characterized in terms of the position of the maximum point of the associated potential.

The proof of Theorem~\ref{thm2} in the case of triangles (see Section~\ref{sec:triangles}) is relatively simple and is based on a straightforward reflection argument inspired by \cite{FraVel19}. However, we also give a second proof which will be extended to the case of quadrilaterals in Section~\ref{sec:squares} (and, hopefully, might work in general for an arbitrary number of sides). This second argument is inspired by an idea of Carrillo, Hittmeir, Volzone, and Yao \cite{CaHiVoYa2016} and is based on a continuous symmetrization (in the spirit of the continuous Steiner symmetrization \cite{Bro95}), see also Figure~\ref{fig:ricIvar}. We show that, if two sides of a triangle have different lengths, then by translating the common vertex parallel to the third side the first variation of the energy is different from zero. In turn, since the criticality condition with respect to this variation can be expressed in terms of the conditions \eqref{eq:sliding} and \eqref{eq:tilting}, we obtain that all sides of a critical triangle have to be equal.

The proof for quadrilaterals exploits the same idea, and uses a continuous symmetrization to prove that the conditions \eqref{eq:sliding} and \eqref{eq:tilting} enforce the property of being equilateral, thus reducing the proof to the class of rhombi; then in a second step we prove that the polygon has to be also equiangular, using a reflection argument. The proof of Theorem~\ref{thm3} follows by the same arguments, with minor changes.

We conjecture that Theorem~\ref{thm2} and Theorem~\ref{thm3} should be true for every fixed number $N\geq3$ of sides, and that a possible strategy for the proof could follow the same ideas sketched above: one should first prove that the polygon is equilateral via continuous symmetrization, and then that it is equiangular via reflection. This strategy is somehow reminiscent of Zenodorus' classical proof of the isoperimetric property of the regular polygons \cite{Leo15}. Notice that a positive answer to this question would also provide an extension of Theorem~\ref{thm1} to the case $N\geq5$. However, the study of the sign of the first variation in the case $N\geq5$ is significantly more involved and seems to require new ideas. This will be the object of future work.

We also remark that, for $N=3$, every triangle satisfies the conditions \eqref{eq:sliding} and \eqref{eq:sliding2}, which therefore do not yield symmetry at all (see Remark~\ref{rmk:triangles}). However, in the case of quadrilaterals both \eqref{eq:sliding} and \eqref{eq:tilting} (or \eqref{eq:sliding2} and \eqref{eq:tilting2}) are required to characterize the square: indeed there exists quadrilaterals different from the square satisfying \eqref{eq:sliding} but not \eqref{eq:tilting} (e.g.\ rhombi), and quadrilaterals different from the square satisfying \eqref{eq:tilting} but not \eqref{eq:sliding} (e.g.\ rectangles).

Finally, we remark on the assumptions we made on the kernel $K$. The regularity assumption \ref{ass:regular} might be relaxed by considering only measurable and nonnegative kernels $K$ and approximating them by a sequence of $C^1$ functions. The assumption \ref{ass:decreasing} is used to obtain the strict monotonicity of the energy with respect to Steiner symmetrizations which, in turn, yields the uniqueness of the maximizer. This assumption is also used to show that certain perturbations of nonregular triangles and quadrilaterals strictly increase the energy in the first order.
The assumption \ref{ass:integrable} guarantees that the energy \eqref{eq:nonlocal_energy} is finite on sets with finite measure.

\medskip

We conclude this introduction by mentioning that our motivation for the study of this problem comes from our recent work \cite{BoCrTo20} on an anisotropic nonlocal isoperimetric problem, recently introduced in \cite{ChNeuTo20} as an extension of the classical liquid drop model of Gamow, in which we considered the volume-constrained minimization of the sum of the nonlocal energy $\nl$ and a crystalline anisotropic perimeter. Due to the presence of a surface tension whose Wulff shape (i.e. the corresponding isoperimetric region) is a convex polygon, it was shown that at least in the small mass regime minimizers of the total energy have a polygonal structure; this naturally led us to the question of characterizing the polygons which are stationary for the nonlocal energy $\nl$.

\medskip\noindent\textit{Structure of the paper.}
The proof of Theorem~\ref{thm1} is given in Section~\ref{sec:Steiner} via Steiner symmetrization. In Section~\ref{sec:overdetermined} we derive the identities \eqref{eq:sliding}, \eqref{eq:tilting}, \eqref{eq:sliding2} and \eqref{eq:tilting2} as stationarity conditions for the nonlocal energy with respect to two particular classes of variations. Finally, Section~\ref{sec:triangles} and Section~\ref{sec:squares} contain the proofs of Theorem~\ref{thm2} and Theorem~\ref{thm3} in the case $N=3$ and $N=4$, respectively.


\section{Maximality of equilateral triangles and squares by Steiner symmetrization}\label{sec:Steiner}

In this section we will give a proof to Theorem~\ref{thm1}. Our proof is based on Steiner symmetrization and a simple argument by P\'{o}lya and Szeg\H{o} which describes two sequences of symmetrizations transforming a given triangle into an equilateral triangle and a given quadrilateral into a square, respectively.

We start by giving the necessary definitions and prove two lemmas regarding the role of Steiner symmetrization on the nonlocal energy $\nl$: in particular, we show that the nonlocal energy is strictly increasing with respect to Steiner symmetrization of a set, unless the set is already symmetric. The \emph{strict} monotonicity of the energy, and the uniqueness of the maximizer, is in turn a consequence of assumption \ref{ass:decreasing}.
Since this monotonicity property is not restricted to dimension 2, in the first part of this section we work in general dimension $d\geq2$, and we replace assumption \ref{ass:integrable} on the kernel by its general version
\begin{equation} \label{ass:integrable_general}
\int_{0}^1 K(r)r^{d-1}\dd r <\infty.
\end{equation}
The proof is essentially contained in \cite[Chapter~3]{LiLo}, but we include the details here to point out the properties that we need. See also \cite{Bu} for details on rearrangement inequalities.

In the following, we denote by $e_1,\ldots,e_d$ the vectors of the canonical basis of $\R^d$. We also denote the generic point of $\R^d\equiv\R^{d-1}\times\R$ by $x=(x',x_d)$.

\begin{definition}\label{defn:symm_rearrange}
Given any measurable set $E \subset \R^d$, its \emph{symmetric rearrangement} is defined as $E^* \defeq B_r$ with $\omega_d r^d = |E|$, where $\omega_d$ denotes the volume of the unit ball in $\R^d$.
\end{definition}

\begin{definition}\label{defn:Steiner_set}
For $E\subset \R^d$ and $x^\pr \in \R^{d-1}$, let $E_{x^\pr} \defeq \bigl\{ x_d \in \R \colon (x^\pr,x_d) \in E \bigr\}$. The \emph{Steiner symmetrization} of $E$ in the direction $e_d$ is defined as
	\[
		E^s \defeq \Bigl\{ (x^\pr,x_d)\in\R^d \colon x^\pr \in \R^{d-1}, \, x_d \in (E_{x^\pr})^* \Bigr\}.
	\]
\end{definition}

Notice that the Steiner symmetrization is a volume preserving operation.

The first lemma shows that Steiner symmetrization of a set $E$ increases its nonlocal energy $\nl$, and it follows from Riesz's rearrangement inequality in one dimension (see \cite[Lemma~3.6]{LiLo}) and Fubini's theorem.

\begin{lemma}\label{lem:Steiner_symm}
Let $E\subset\R^d$ be a measurable set with finite measure. Then
	\[
		\nl(E) \leq \nl(E^s).
	\]
\end{lemma}

\begin{proof}
We first prove the following property: given three measurable sets $F$, $G$, and $H \subset \Rd$ with finite measure, we have
\begin{equation} \label{proof:steiner}
\mathcal{I} (F,G,H) \leq \mathcal{I} (F^s,G^s,H^s),
\end{equation}
where $\mathcal{I}(F,G,H) \defeq \int_{\Rd}\!\int_{\Rd} \chi_F(x) \chi_G(x-y) \chi_H(y) \dd x \dd y$. Indeed, by Fubini's theorem
	\begin{align*}
		\mathcal{I}(F,G,H) &= \int_{\R^{d-1}}\!\int_{\R^{d-1}} \int_\R\!\int_\R \chi_F(x^\pr,x_d) \chi_G(x^\pr-y^\pr,x_d-y_d)\chi_H(y^\pr,y_d)\dd x_d \dd y_d \dd x^\pr \dd y^\pr \\
										 &= \int_{\R^{d-1}}\!\int_{\R^{d-1}} \int_\R\!\int_\R \chi_{F_{x^\pr}}(x_d) \chi_{G_{x^\pr-y^\pr}}(x_d-y_d)\chi_{H_{y^\pr}}(y_d)\dd x_d \dd y_d \dd x^\pr \dd y^\pr \\
										 &\leq \int_{\R^{d-1}}\!\int_{\R^{d-1}} \int_\R\!\int_\R \chi_{(F_{x^\pr})^*}(x_d) \chi_{(G_{x^\pr-y^\pr})^*}(x_d-y_d)\chi_{(H_{y^\pr})^*}(y_d)\dd x_d \dd y_d \dd x^\pr \dd y^\pr \\
										 &= \int_{\Rd}\!\int_{\Rd} \chi_{F^s}(x)\chi_{G^s}(x-y)\chi_{H^s}(y) \dd x \dd y = \mathcal{I}(F^s,G^s,H^s),
	\end{align*}
where the inequality follows from the one dimensional Riesz's rearrangement inequality.

Now, since the kernel $K$ is strictly decreasing, for any $t>0$ there exists $r(t)>0$ such that $\{x\in\Rd \colon K(|x|) > t\}=B_{r(t)}$. Using the layer cake formula (see \cite[Theorem~1.13]{LiLo}) and Fubini's theorem, we can rewrite the nonlocal energy as
	\begin{align*}
		\nl(E)  &= \int_{\Rd}\!\int_{\Rd} \chi_E(x) K(|x-y|) \chi_E(y) \dd x \dd y \\
				   &= \int_{\Rd}\!\int_{\Rd} \chi_E(x) \left(\int_0^\infty \chi_{\{K>t\}}(|x-y|) \dd t \right) \chi_E(y) \dd x \dd y \\
				   &= \int_0^\infty \left( \int_{\Rd}\!\int_{\Rd} \chi_E(x) \chi_{B_{r(t)}}(x-y) \chi_E(y) \dd x \dd y \right) \dd t.
	\end{align*}
Then \eqref{proof:steiner} implies that
	\[
		\int_{\Rd}\!\int_{\Rd} \chi_E(x) \chi_{B_{r(t)}}(x-y) \chi_E(y) \dd x \dd y \leq \int_{\Rd}\!\int_{\Rd} \chi_{E^s}(x) \chi_{B_{r(t)}}(x-y) \chi_{E^s}(y) \dd x \dd y.
	\]
Hence, rewriting the energy of $E^s$ using Fubini's theorem and the layer cake representation as above, we get $\nl(E) \leq \nl(E^s)$.
\end{proof}

\begin{remark} \label{rmk:finiteness}
Notice that for every measurable set $E\subset\R^d$ with finite measure, in view of the assumption \eqref{ass:integrable_general} and of the monotonicity of $K$, the potential $v_E$ defined in \eqref{eq:potential} is a bounded function:
\begin{align*}
v_E(x) & = \int_{E\cap B_1(x)}K(|x-y|)\dd y + \int_{E\setminus B_1(x)}K(|x-y|)\dd y \\
& \leq \int_{B_1}K(|y|)\dd y + K(1)|E\setminus B_1(x)| \\
& \leq d\omega_d\int_0^1 K(r)r^{d-1}\dd r + K(1)|E| =: C(d,K,|E|)<\infty.
\end{align*}
In turn, the energy of $E$ is finite: $\nl(E)=\int_E v_E(x)\dd x \leq C(d,K,|E|)|E|$.
\end{remark}

The next lemma shows that if a set and its Steiner symmetral have the same nonlocal energy, then they are translates of each other almost everywhere.

\begin{lemma}\label{lem:Steiner_uniq}
Let $E\subset\R^d$ be a measurable set with finite measure. Then $\nl(E)=\nl(E^s)$ only if $|E \triangle (E^s + y_0)|=0$ for some $y_0 \in \Rd$, where $\triangle$ denotes the symmetric difference of two sets.
\end{lemma}

\begin{proof}
For dimension $d=1$, the result follows from \cite[Theorem~3.9]{LiLo}. Let $\ka(x) \defeq K(|x|)$ for $x\in\Rd$. For $d>1$, we note that, by Fubini's theorem, the equality $\nl(E) = \nl(E^s)$ is equivalent to
	\begin{multline} \nonumber
		\int_{\R^{d-1}}\! \int_{\R^{d-1}} \int_\R\!\int_\R \chi_E(x^\pr,x_d)\chi_E(y^\pr,y_d)\, \ka(x^\pr - y^\pr, x_d-y_d) \dd x_d \dd y_d \dd x^\pr \dd y^\pr \\ 
			= \int_{\R^{d-1}}\! \int_{\R^{d-1}} \int_\R\!\int_\R \chi_{E^s}(x^\pr,x_d)\chi_{E^s}(y^\pr,y_d)\, \ka(x^\pr - y^\pr, x_d-y_d) \dd x_d \dd y_d \dd x^\pr \dd y^\pr.
	\end{multline}

Since $\chi_{E^s}(x^\pr,x_d) = \chi_{(E_{x^\pr})^*}(x_d)$, defining
	\[
		\mathcal{I}^{1} \bigl( E_{x^\pr}, \ka(x^\pr-y^\pr, \,\cdot\,),E_{y^\pr} \bigr) \defeq \int_\R\!\int_\R \chi_{E_{x^\pr}}(x_d)\chi_{E_{y^\pr}}(y_d)\, \ka(x^\pr - y^\pr, x_d-y_d) \dd x_d \dd y_d 
	\]
the above equation becomes
	\begin{multline} \nonumber
		\int_{\R^{d-1}}\! \int_{\R^{d-1}} \mathcal{I}^{1} \bigl( E_{x^\pr}, \ka(x^\pr-y^\pr, \,\cdot\,),E_{y^\pr} \bigr) \dd x^\pr \dd y^\pr \\=  \int_{\R^{d-1}}\! \int_{\R^{d-1}} \mathcal{I}^{1} \bigl( (E_{x^\pr})^*, \ka(x^\pr-y^\pr, \,\cdot\,),(E_{y^\pr})^* \bigr) \dd x^\pr \dd y^\pr.
	\end{multline}
In turn, since by Riesz's rearrangement inequality (cf. \cite[Theorem~3.7]{LiLo})
	\[
		\mathcal{I}^{1} \bigl( E_{x^\pr}, \ka(x^\pr-y^\pr, \,\cdot\,),E_{y^\pr} \bigr) \leq \mathcal{I}^{1} \bigl( (E_{x^\pr})^*, \ka(x^\pr-y^\pr, \,\cdot\,),(E_{y^\pr})^* \bigr),
	\]
we get that
	\[
		\mathcal{I}^{1} \bigl( E_{x^\pr}, \ka(x^\pr-y^\pr, \,\cdot\,),E_{y^\pr} \bigr) = \mathcal{I}^{1} \bigl( (E_{x^\pr})^*, \ka(x^\pr-y^\pr, \,\cdot\,),(E_{y^\pr})^* \bigr),
	\]
for a.e. $x^\pr$, $y^\pr \in \R^{d-1}$. This implies, by the one dimensional result, that $E_{x^\pr}$ and $E_{y^\pr}$ are both intervals centered at the same point for a.e. $(x^\pr,y^\pr) \in \R^{d-1} \times \R^{d-1}$. Moreover, this point is independent of $(x^\pr,y^\pr)$ as we can repeat the argument for any $(x^\pr,\tilde{y}^\pr)$ with $\tilde{y}^{\pr}\in \R^{d-1}$ and obtain that the centers of $E_{x^\pr}$ and $E_{\tilde{y}^\pr}$ coincide. Therefore the set $E$, after possibly a translation in the $x_d$ direction, is Steiner symmetric up to a set of measure zero, i.e., $|E \triangle (E^s+y_0)|=0$ for some $y_0 \in \Rd$.
\end{proof}

We are now ready to prove our first main result which relies on an argument by P\'{o}lya and Szeg\H{o} that we detail here.

\begin{proof}[Proof of Theorem~\ref{thm1}]
Let $\mathcal{P}_0\in\pol_3$ be an arbitrary triangle with $|\p_0|=1$. Following \cite[Section~7.4]{PoSz51}, we will describe an infinite sequence of Steiner symmetrizations of $\mathcal{P}_0$ which will transform it into an equilateral triangle. To this end, let $2a_0$ be the length of one of the sides of $\mathcal{P}_0$. Then the corresponding altitude perpendicular to this side has length $a_0^{-1}$. By Steiner symmetrization of $\mathcal{P}_0$ in the direction of this side, we obtain an isosceles triangle $\mathcal{P}_1$ where the length of equal sides is $a_1=(a_0^2+a_0^{-2})^{1/2}$. Next we symmetrize $\mathcal{P}_1$ in the direction of one of the equal sides to obtain another isosceles triangle $\mathcal{P}_2$ with equal sides of length $a_2=(a_1^2/4+4/a_1^2)^{1/2}$. Repeating this process, we see that the length of the equal sides of the isosceles triangle $\mathcal{P}_n$ is given recursively by
$$
a_n = \sqrt{ \frac{a_{n-1}^2}{4} + \frac{4}{a_{n-1}^{2}}}
$$
for $n\geq2$. It can be checked that the sequence $(a_n^2)_n$ is a Cauchy sequence, by showing that $|a_{n+2}^2-a_{n+1}^2|/|a_{n+1}^2-a_n^2|\leq\frac34$; therefore, taking the limit $n\to\infty$ we see that $a_n\to 2/\sqrt[4]{3}$, and since in each iteration the area of $\p_n$ is one, in the limit, we obtain that all three sides are of length $2/\sqrt[4]{3}$.

Now, suppose $\mathcal{P}_0\in\pol_4$ is an arbitrary quadrilateral with $|\p_0|=1$. Symmetrizing $\mathcal{P}_0$ in the direction of one of its diagonals we obtain a kite, $\mathcal{P}_1$ (that is, a quadrilateral with a diagonal as axis of symmetry). If $\p_0$ is not convex, we symmetrize in the direction of its internal diagonal, so that in any case $\p_1$ is a convex quadrilateral. Next, we symmetrize $\mathcal{P}_1$ in the direction of its axis of symmetry and obtain a rhombus, $\mathcal{P}_2$. Let $a_2$ be the side length of $\p_2$. Symmetrizing $\mathcal{P}_2$ in the direction of one of its sides we get a rectangle $\mathcal{P}_3$ such that its longer side has length $a_3=a_2$. Symmetrizing $\mathcal{P}_3$ in the direction of one of its diagonals we obtain another rhombus, $\mathcal{P}_4$, with side length $a_4=\big( a_3^2/(a_3^4+1) + (a_3^4+1)/(4a_3^2) \big)^{1/2}$. Continuing this process we will obtain a sequence of quadrilaterals such that $\p_n$ is a rhombus for $n$ even, and a rectangle for $n$ odd. If $a_n$ denotes the side length of $\p_n$ ($n$ even) or the length of the longer side ($n$ odd), we have by construction
	\[
		a_{2n+1}=a_{2n}, \qquad a_{2n} = \sqrt{\frac{a_{2n-1}^2}{a_{2n-1}^4+1} + \frac{a_{2n-1}^4+1}{4a_{2n-1}^2}} \qquad \text{ for } n \geq 2,
	\]
recursively. Since by construction $a_n\geq1$, it can be checked that the sequence $(a_n)_n$ is monotone decreasing. Therefore, taking the limit $n\to\infty$ we get that $a_n \to 1$; hence, in the limit successive symmerizations of $\mathcal{P}_0$ yield a square.

Since, by Lemma~\ref{lem:Steiner_symm}, Steiner symmetrization increases the nonlocal energy $\nl$, we obtain that among the classes $\pol_3$ and $\pol_4$ an equilateral triangle and a square maximize $\nl$, respectively. The uniqueness of the maximizer in each class, up to rigid movements, follows from Lemma~\ref{lem:Steiner_uniq}.
\end{proof}


\section{Stationarity conditions: sliding and tilting}\label{sec:overdetermined}

We derive the stationarity conditions for the nonlocal energy \eqref{eq:nonlocal_energy} under an area or a perimeter constraint, with respect to two particular classes of perturbations of a polygon $\mathcal{P}\in\pol_N$, obtained by sliding one side parallel to itself, or tilting one side around its midpoint. In the following, we first assume that $\mathcal{P}\in\pol_N$ is a given \emph{convex} polygon with $N\geq3$ vertices $P_1,\ldots,P_N$. We choose to present the classes of perturbations and obtain stationarity conditions for convex polygons first in order to keep the presentation simple and the proofs clear. These classes extend easily to nonconvex polygons albeit the extension for tilting one side requires the introduction of new notation, and the \emph{same} stationary conditions are satisfied by a nonconvex $\mathcal{P}$. We present this extension to nonconvex polygons in a separate subsection. We consider the following two families of one-parameter deformations. 

\begin{definition}[Sliding of one side] \label{def:sliding}
Fix a side $\side{P_i}{P_{i+1}}$, $i\in\{1,\ldots,N\}$. For $t\in\R$ with $|t|$ sufficiently small, we define the polygon $\mathcal{P}_t\in\pol_N$ with vertices $P_1^t,\ldots,P_N^t$ obtained as follows (see Figure~\ref{fig:sliding}):
\begin{enumerate}[label = (\roman*)]
\item all vertices except $P_i$ and $P_{i+1}$ are fixed, i.e.\ $P_j^t\defeq P_j$ for all $j\in\{1,\ldots N\}\setminus\{i,i+1\}$;
\item the vertices $P_i^t$ and $P_{i+1}^t$ lie on the lines containing $\side{P_{i-1}}{P_i}$ and $\side{P_{i+1}}{P_{i+2}}$, respectively;
\item the side $\side{P_i^t}{P_{i+1}^t}$ is parallel to $\side{P_i}{P_{i+1}}$ and at a distance $|t|$ from $\side{P_i}{P_{i+1}}$, in the direction of $\nu_i$ if $t>0$ and in the direction of $-\nu_i$ if $t<0$.
\end{enumerate}
Explicitly:
\begin{equation*}
P_i^t\defeq P_i + \frac{t}{\sin\theta_i}\frac{P_i-P_{i-1}}{|P_i-P_{i-1}|}\,, \qquad P_{i+1}^t\defeq P_{i+1} + \frac{t}{\sin\theta_{i+1}}\frac{P_{i+1}-P_{i+2}}{|P_{i+1}-P_{i+2}|}\,.
\end{equation*}
\end{definition}

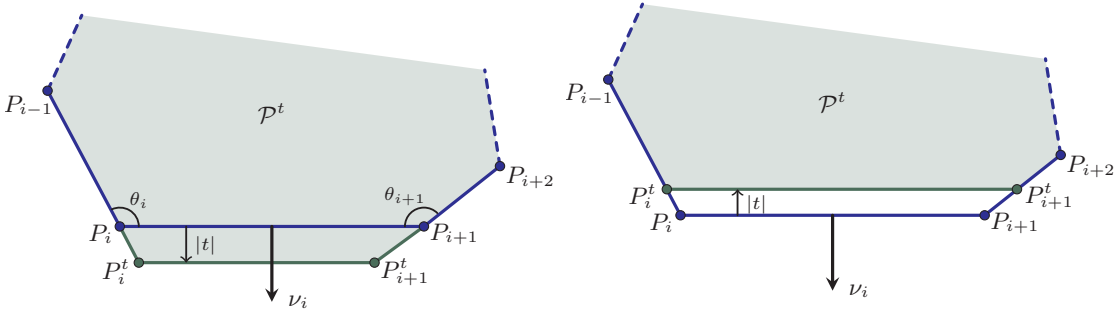
\begin{figure}[ht]
\definecolor{qqqqff}{rgb}{0,0,1}
\definecolor{qqwuqq}{rgb}{0,0.4,0}
\begin{tikzpicture}[scale=0.5,line cap=round,line join=round]
\clip(-7,-2.2) rectangle (7.5,6);
\fill[line width=0pt,fill=qqwuqq,fill opacity=0.2] (-3.5,-0.96) -- (2.7,-0.96) -- (6,1.6) -- (5.6,4.15) -- (-5,5.6) -- (-5.88,3.6) -- cycle;
\draw [line width=1.2pt,dash pattern=on 3pt off 3pt,color=qqqqff] (-5.9,3.6)-- (-5,5.6);
\draw [line width=1.2pt,dash pattern=on 3pt off 3pt,color=qqqqff] (6,1.6)-- (5.6,4.15);
\draw [line width=1.2pt,color=qqqqff] (-5.9,3.6)-- (-4,0);
\draw [line width=1.2pt,color=qqqqff] (-4,0)-- (4,0);
\draw [line width=1.2pt,color=qqqqff] (4,0)-- (6,1.6);
\draw [line width=1.2pt,color=qqwuqq] (-4,0)-- (-3.5,-0.96);
\draw [line width=1.2pt,color=qqwuqq] (4,0)-- (2.7,-0.96);
\draw [line width=1.2pt,color=qqwuqq] (2.7,-0.96)-- (-3.5,-0.96);
\draw [-stealth,line width=1.2pt] (0,0) -- (0,-2);
\draw [->,line width=0.7pt] (-2.253649257845337,0)-- (-2.253649257845337,-0.96);
\begin{footnotesize}
\draw [fill=qqqqff] (-4,0) circle (3.5pt);
\draw (-4.5,-0.2) node {$P_i$};
\draw [fill=qqqqff] (4,0) circle (3.5pt);
\draw (4.8,-0.2) node {$P_{i+1}$};
\draw [fill=qqqqff] (-5.9,3.6) circle (3.5pt);
\draw (-6.4,3.2) node {$P_{i-1}$};
\draw [fill=qqqqff] (6,1.6) circle (3.5pt);
\draw (6.8,1.3) node {$P_{i+2}$};
\draw [fill=qqwuqq] (-3.5,-0.96) circle (3.5pt);
\draw (-4.1,-1.2) node {$P_i^t$};
\draw [fill=qqwuqq] (2.7,-0.96) circle (3.5pt);
\draw (3.5,-1.2) node {$P_{i+1}^t$};
\draw[color=black] (0.7,-2) node {$\nu_i$};
\draw[color=black] (-1.75,-0.5) node {\tiny $|t|$};
\draw (0,3) node {$\p^t$};
\draw [-,line width=0.7pt] (-3.5,0) arc (0:115:0.5);
\draw (-3.5,0.7) node {\tiny $\theta_i$};
\draw [-,line width=0.7pt] (3.5,0) arc (180:45:0.5);
\draw (3.5,0.8) node {\tiny $\theta_{i+1}$};
\end{footnotesize}
\end{tikzpicture}
\begin{tikzpicture}[scale=0.5,line cap=round,line join=round]
\clip(-7,-2.5) rectangle (7.5,6);
\fill[line width=0pt,fill=qqwuqq,fill opacity=0.2] (-4.36,0.69) -- (4.85,0.69) -- (6,1.6) -- (5.6,4.15) -- (-5,5.6) -- (-5.88,3.6) -- cycle;
\draw [line width=1.2pt,dash pattern=on 3pt off 3pt,color=qqqqff] (-5.9,3.6)-- (-5,5.6);
\draw [line width=1.2pt,dash pattern=on 3pt off 3pt,color=qqqqff] (6,1.6)-- (5.6,4.15);
\draw [line width=1.2pt,color=qqqqff] (-5.9,3.6)-- (-4,0);
\draw [line width=1.2pt,color=qqqqff] (-4,0)-- (4,0);
\draw [line width=1.2pt,color=qqqqff] (4,0)-- (6,1.6);
\draw [line width=1.2pt,color=qqwuqq] (4.85,0.69)-- (-4.36,0.69);
\draw [-stealth,line width=1.2pt] (0,0) -- (0,-2);
\draw [->,line width=0.7pt] (-2.5,0)-- (-2.5,0.69);
\begin{footnotesize}
\draw [fill=qqqqff] (-4,0) circle (3.5pt);
\draw (-4.5,-0.2) node {$P_i$};
\draw [fill=qqqqff] (4,0) circle (3.5pt);
\draw (4.8,-0.2) node {$P_{i+1}$};
\draw [fill=qqqqff] (-5.9,3.6) circle (3.5pt);
\draw (-6.4,3.2) node {$P_{i-1}$};
\draw [fill=qqqqff] (6,1.6) circle (3.5pt);
\draw (6.8,1.3) node {$P_{i+2}$};
\draw [fill=qqwuqq] (-4.36,0.69) circle (3.5pt);
\draw (-5,0.5) node {$P_{i}^t$};
\draw [fill=qqwuqq] (4.85,0.69) circle (3.5pt);
\draw (5.7,0.5) node {$P_{i+1}^t$};
\draw[color=black] (0.7,-2) node {$\nu_i$};
\draw[color=black] (-2,0.3) node {\tiny $|t|$};
\draw (0,3) node {$\p^t$};
\end{footnotesize}
\end{tikzpicture}
\caption{A polygon $\p$ and its variation $\p^t$ (shaded region) as in Definition~\ref{def:sliding}, obtained by sliding the side $\side{P_i}{P_{i+1}}$ in the normal direction at a distance $|t|$: the case $t>0$ (left) and $t<0$ (right).}\label{fig:sliding}
\end{figure}

\begin{definition}[Tilting of one side] \label{def:tilting}
Fix a side $\side{P_i}{P_{i+1}}$, $i\in\{1,\ldots,N\}$. For $t\in\R$ with $|t|$ sufficiently small, we define the polygon $\mathcal{P}_t\in\pol_N$ with vertices $P_1^t,\ldots,P_N^t$ obtained as follows (see Figure~\ref{fig:tilting}):
\begin{enumerate}[label = (\roman*)]
\item all vertices except $P_i$ and $P_{i+1}$ are fixed, i.e.\ $P_j^t\defeq P_j$ for all $j\in\{1,\ldots N\}\setminus\{i,i+1\}$;
\item the vertices $P_i^t$ and $P_{i+1}^t$ lie on the lines containing $\side{P_{i-1}}{P_i}$ and $\side{P_{i+1}}{P_{i+2}}$, respectively;
\item the line containing $\side{P_i^t}{P_{i+1}^t}$ is obtained by rotating the line containing $\side{P_i}{P_{i+1}}$ around the midpoint $M_i$ of $\side{P_i}{P_{i+1}}$ by an angle $t$;
\item the direction of rotation is such that, for $t>0$, the point $P_{i+1}^t$ belongs to the segment $\side{P_{i+1}}{P_{i+2}}$, while for $t<0$ the point $P_i^t$ belongs to the segment $\side{P_{i-1}}{P_i}$.
\end{enumerate}
Explicitly:
\begin{equation*}
P_i^t\defeq P_i +\frac{\ell_i\sin t}{2\sin(\theta_i-t)}\frac{P_{i}-P_{i-1}}{|P_{i}-P_{i-1}|}\,, \qquad P_{i+1}^t\defeq P_{i+1}-\frac{\ell_i\sin t}{2\sin(\theta_{i+1}+t)}\frac{P_{i+1}-P_{i+2}}{|P_{i+1}-P_{i+2}|}\,.
\end{equation*}
\end{definition}

\begin{figure}[ht]
\definecolor{qqqqff}{rgb}{0,0,1}
\definecolor{qqwuqq}{rgb}{0,0.4,0}
\begin{tikzpicture}[scale=0.6,line cap=round,line join=round]
\clip(-7,-1.8) rectangle (7.5,6);
\fill[line width=0pt,fill=qqwuqq,fill opacity=0.2] (-3.45,-1.04) -- (5.28,1.58) -- (6,2.48) -- (5.57,4.17) -- (-5,5.67) -- (-5.94,3.66) -- cycle;
\draw [line width=1.2pt,dash pattern=on 3pt off 3pt,color=qqqqff] (-5.94,3.66)-- (-5,5.67);
\draw [line width=1.2pt,dash pattern=on 3pt off 3pt,color=qqqqff] (6,2.48)-- (5.57,4.17);
\draw [line width=1.2pt,color=qqqqff] (-5.94,3.66)-- (-4,0);
\draw [line width=1.2pt,color=qqqqff] (-4,0)-- (4,0);
\draw [line width=1.2pt,color=qqqqff] (4,0)-- (6,2.48);
\draw [line width=1.2pt,color=qqwuqq] (-4,0)-- (-3.449980437567673,-1.0361968117094946);
\draw [line width=1.2pt,color=qqwuqq] (-3.449980437567673,-1.0361968117094946)-- (5.2792724920825576,1.585622128425266);
\draw [line width=1.2pt,color=qqwuqq] (5.2792724920825576,1.585622128425266)-- (6,2.48);
\begin{footnotesize}
\draw [fill=qqqqff] (-4,0) circle (3.5pt);
\draw (-4.4,-0.2) node {$P_i$};
\draw [fill=qqqqff] (4,0) circle (3.5pt);
\draw (4.7,-0.2) node {$P_{i+1}$};
\draw [fill=qqqqff] (-5.94,3.66) circle (3.5pt);
\draw (-6.3,3.3) node {$P_{i-1}$};
\draw [fill=qqqqff] (6,2.48) circle (3.5pt);
\draw (6.7,2.4) node {$P_{i+2}$};
\draw [fill] (0,0) circle (3.5pt);
\draw (0.3,-0.4) node {$M_i$};
\draw [fill=qqwuqq] (5.28,1.58) circle (3.5pt);
\draw (5.9,1.4) node {$P_{i+1}^t$};
\draw [fill=qqwuqq] (-3.45,-1.04) circle (3.5pt);
\draw (-3.9,-1.4) node {$P_i^t$};
\draw (0,3) node {$\p^t$};
\draw [->,line width=0.7pt] (1.8,0) arc (0:17:1.8);
\draw (2.1,0.3) node {\tiny $t$};
\end{footnotesize}
\end{tikzpicture}
\caption{A polygon $\p$ and its variation $\p^t$ (shaded region) as in Definition~\ref{def:tilting}, obtained by tilting the side $\side{P_i}{P_{i+1}}$ around its midpoint $M_i$ by an angle $t>0$.}\label{fig:tilting}
\end{figure}
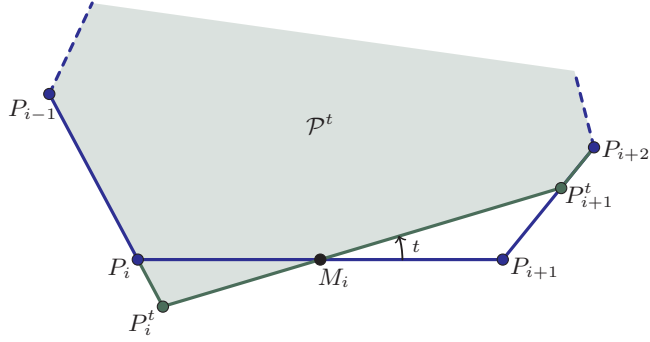

We next show in Theorem~\ref{thm:stationarity} that the equations \eqref{eq:sliding}-\eqref{eq:tilting} and \eqref{eq:sliding2}--\eqref{eq:tilting2} are the stationarity conditions for the nonlocal energy under an area or a perimeter constraint respectively, with respect to the variations considered in Definitions~\ref{def:sliding} and \ref{def:tilting}. To do so, we need to compute the first variation of the nonlocal energy \eqref{eq:nonlocal_energy}, of the area and of the perimeter of a polygon $\mathcal{P}$ with respect to these two classes of perturbations. The computation is based on the following first variation formula for the nonlocal energy with respect to a general perturbation. The derivation, which is valid in any dimension $d\geq2$, replacing the assumption \ref{ass:integrable} by \eqref{ass:integrable_general}, is relatively standard and is presented in the Appendix~\ref{sec:Appendix} for completeness.

\begin{proposition}[First variation of $\nl$] \label{prop:firstvar}
Let $E\subset\R^2$ be a bounded open set with piecewise smooth boundary. Let $\Phi:\R^2\times[-\bar{t},\bar{t}]\to\R^2$, for $\bar{t}>0$, be a flow of class $C^2$ such that $\Phi(x,0)=x$. Then
\begin{equation} \label{eq:firstvar}
\frac{\dd}{\dd t}\Big|_{t=0} \nl(\Phi_t(E)) = 2\int_{\partial E}v_E(x)\,X(x)\cdot\nu_E(x)\dd\Hone(x),
\end{equation}
where $X(x)\defeq\frac{\partial\Phi(x,t)}{\partial t}|_{t=0}$ is the initial velocity, $v_E$ is the potential of $E$ defined in \eqref{eq:potential}, and $\nu_E$ is the exterior unit normal on $\partial E$.
\end{proposition}

\begin{definition}[Stationarity] \label{def:stationarity}
Let $\p\in\pol_N$ and let $\{\p_t\}_t$ be a one-parameter deformation of $\p$, such as those considered before. We define an area-preserving variation and a perimeter-preserving variation, rescaling $\p_t$ by
\begin{equation} \label{eq:stationarya}
\mathcal{Q}_t \defeq \lambda_t\p_t \qquad\text{where }\lambda_t\defeq\biggl(\frac{|\p|}{|\p_t|}\biggr)^{\frac12},
\end{equation}
\begin{equation} \label{eq:stationaryb}
\mathcal{R}_t \defeq \mu_t\p_t \qquad\text{where }\mu_t\defeq \frac{\per(\p)}{\per(\p_t)},
\end{equation}
respectively, so that $|\mathcal{Q}_t|=|\p|$ and $\per(\mathcal{R}_t)=\per(\p)$ for all $t$. We say that $\p$ is \emph{stationary with respect to the variation $\{\p_t\}_t$ under area constraint} if
\begin{equation} \label{eq:stationaryc}
\frac{\dd}{\dd t}\Big|_{t=0} \nl(\mathcal{Q}_t)  = 0,
\end{equation}
and that $\p$ is \emph{stationary with respect to the variation $\{\p_t\}_t$ under perimeter constraint} if
\begin{equation} \label{eq:stationaryd}
\frac{\dd}{\dd t}\Big|_{t=0} \nl(\mathcal{R}_t)  = 0.
\end{equation}
\end{definition}

\begin{theorem}[Stationarity conditions] \label{thm:stationarity}
A polygon $\p\in\pol_N$ is stationary with respect to the sliding variation as in Definition~\ref{def:sliding} on the $i$-th side, for $i\in\{1,\ldots,N\}$,
\begin{enumerate}
\item under area constraint if and only if
\begin{equation}
\frac{1}{\ell_i}\int_{\side{P_i}{P_{i+1}}} v_{\p}(x)\dd\Hone(x) = \frac{\sigma}{2|\p|},
\end{equation}
\item under perimeter constraint if and only if
\begin{equation}
\int_{\side{P_i}{P_{i+1}}} v_{\p}(x)\dd\Hone(x) = \frac{\sigma}{\per(\p)}\bigl(\psi(\theta_i)+\psi(\theta_{i+1})\bigr).
\end{equation}
\end{enumerate}
A polygon $\p\in\pol_N$ is stationary with respect to the tilting variation as in Definition~\ref{def:tilting} on the $i$-th side, for $i\in\{1,\ldots,N\}$,
\begin{enumerate}
\item under area constraint if and only if
\begin{equation}
\int_{\side{P_i}{M_i}} v_{\p}(x)|x-M_i|\dd\Hone(x) = \int_{\side{P_{i+1}}{M_i}} v_{\p}(x)|x-M_i|\dd\Hone(x),
\end{equation}
\item under perimeter constraint if and only if
\begin{multline}
\int_{\side{P_i}{M_i}} v_{\p}(x)|x-M_i|\dd\Hone(x) - \int_{\side{P_{i+1}}{M_i}} v_{\p}(x)|x-M_i|\dd\Hone(x) \\
= \frac{\sigma\ell_i}{2\per(\p)}\bigl(\psi(\theta_i)-\psi(\theta_{i+1})\bigr).
\end{multline}
\end{enumerate}
In the previous equations the constant $\sigma$ is independent of $i$ and defined as
\begin{equation} \label{lagrangemult}
\sigma \defeq \int_{\partial \p}v_{\p}(x)\,x\cdot\nu_\p(x)\dd\Hone(x).
\end{equation}
\end{theorem}

\begin{proof}
Let $\{\Phi_t\}_t$ be a flow such that $\p_t=\Phi_t(\p)$ and let $X(x)\defeq\frac{\partial\Phi(x,t)}{\partial t}|_{t=0}$ be the initial velocity. We compose the flow $\{\Phi_t\}_t$ with a rescaling which restores the area or the perimeter constraint: more precisely, we define
$$
\Psi(x,t) \defeq \sigma_t\Phi(x,t)
$$
where $\sigma_t$ is either equal to $\lambda_t$ (defined in \eqref{eq:stationarya}) or to $\mu_t$ (defined in \eqref{eq:stationaryb}). Notice that $\Psi_t(\p)=\sigma_t\Phi_t(\p)=\sigma_t\p_t$, therefore $\Psi_t(\p)=\mathcal{Q}_t$ if $\sigma_t=\lambda_t$, and $\Psi_t(\p)=\mathcal{R}_t$ if $\sigma_t=\mu_t$. The initial velocity of the flow $\{\Psi_t\}_t$ is given by
\begin{equation*}
Y(x) \defeq\frac{\partial\Psi(x,t)}{\partial t}\Big|_{t=0} = X(x) + \frac{\dd \sigma_t}{\dd t}\Big|_{t=0}x .
\end{equation*}
Then by Proposition~\ref{prop:firstvar} we obtain
\begin{equation*}
\begin{split}
\frac{\dd}{\dd t}\Big|_{t=0} \nl(\Psi_t(\p))
& = 2\int_{\partial \p}v_{\p}(x)\,Y(x)\cdot\nu_\p(x)\dd\Hone(x) \\
& = 2\int_{\partial \p}v_{\p}(x)\,X(x)\cdot\nu_\p(x)\dd\Hone(x) + 2\,\frac{\dd \sigma_t}{\dd t}\Big|_{t=0}\int_{\partial \p}v_{\p}(x)\,x\cdot\nu_\p(x)\dd\Hone(x) \\
& = \frac{\dd}{\dd t}\Big|_{t=0} \nl(\p_t) + 2\,\frac{\dd \sigma_t}{\dd t}\Big|_{t=0} \int_{\partial \p}v_{\p}(x)\,x\cdot\nu_\p(x)\dd\Hone(x).
\end{split}
\end{equation*}
Therefore the stationarity conditions \eqref{eq:stationaryc} and \eqref{eq:stationaryd} with respect to the perturbation $\{\p_t\}_t$, under area or perimeter constraint, are equivalent to
\begin{equation} \label{proof:stationary1}
\frac{\dd}{\dd t}\Big|_{t=0} \nl(\p_t) = -2\sigma \frac{\dd \lambda_t}{\dd t}\Big|_{t=0} = \frac{\sigma}{|\p|}\frac{\dd}{\dd t}\Big|_{t=0} |\mathcal{P}_t|
\end{equation}
and
\begin{equation} \label{proof:stationary2}
\frac{\dd}{\dd t}\Big|_{t=0} \nl(\p_t) = -2\sigma \frac{\dd \mu_t}{\dd t}\Big|_{t=0} = \frac{2\sigma}{\per(\p)}\frac{\dd}{\dd t}\Big|_{t=0} \per(\p_t)
\end{equation}
respectively. Thus, to obtain the conditions in the statement, it is now sufficient to insert 
in \eqref{proof:stationary1} and \eqref{proof:stationary2} the first variation formulas for the nonlocal energy, the area, and the perimeter of a polygon with respect to the sliding and tilting variations, that we now compute.

\medskip
\noindent\emph{Sliding first variation.} The flow $\{\Phi_t\}_t$ which induces the perturbation as in Definition~\ref{def:sliding}, obtained by sliding the side $\side{P_i}{P_{i+1}}$ parallel to itself, obeys $(\Phi_t(x)-x)\cdot\nu_i=t$ for all $x\in\side{P_i}{P_{i+1}}$; therefore its initial velocity has normal component
\begin{equation*}
X\cdot\nu_i = 1 \qquad\text{on }\side{P_i}{P_{i+1}}
\end{equation*} 
and $X\cdot\nu_j=0$ for all $j\neq i$. Hence, by Proposition~\ref{prop:firstvar} we obtain
\begin{equation} \label{firstvar_sliding1}
\frac{\dd}{\dd t}\Big|_{t=0} \nl(\mathcal{P}_t) = 2\int_{\side{P_i}{P_{i+1}}}v_{\p}(x)\dd\Hone(x).
\end{equation}
The first variations of the area and of the perimeter are computed in \cite[Lemma~2.7]{FraVel19} in the case of a triangle, but the proof is obviously the same for a general polygon, and follows, via elementary geometric arguments, from the identities
$$
|\p_t| = |\p| + \ell_i t + o(t), \qquad \per(\p_t)= \per(\p) +t\bigl(\psi(\theta_i)+\psi(\theta_{i+1})\bigr)
$$
as $t\to0$, where $\psi$ is the function defined in \eqref{eq:cotangent}. Therefore, we obtain
\begin{equation} \label{firstvar_sliding2}
\frac{\dd}{\dd t}\Big|_{t=0} |\mathcal{P}_t| = \ell_i, \qquad \frac{\dd}{\dd t}\Big|_{t=0} \per(\p_t) = \psi(\theta_i) + \psi(\theta_{i+1}).
\end{equation}

\medskip
\noindent\emph{Tilting first variation.} 
We can explicitly write a flow $\{\Phi_t\}_t$ which induces the perturbation  as in Definition~\ref{def:tilting}, obtained by tilting the side $\side{P_i}{P_{i+1}}$ with respect to its midpoint $M_i$.  On the side $\side{P_i}{P_{i+1}}$ it is given by
\begin{equation*}
\Phi_t(x) =
\begin{cases}
\displaystyle x - \frac{\sin t}{\sin(\theta_i-t)}|x-M_i|\tau_i & \text{if }x\in\side{P_i}{M_i}, \\[10pt]
\displaystyle x + \frac{\sin t}{\sin(\theta_{i+1}+t)}|x-M_i|\tau_{i+1} & \text{if }x\in\side{M_i}{P_{i+1}},
\end{cases}
\end{equation*}
where $\tau_i=\frac{1}{\ell_{i-1}}(P_{i-1}-P_i)$ and $\tau_{i+1}=\frac{1}{\ell_{i+1}}(P_{i+2}-P_{i+1})$ are the unit vectors parallel to the sides $\side{P_{i-1}}{P_i}$ and $\side{P_{i+1}}{P_{i+2}}$, respectively.
Then the normal component of the initial velocity is
\begin{equation*}
X(x)\cdot\nu_i =
\begin{cases}
\displaystyle -\frac{|x-M_i|}{\sin\theta_{i}} \, \tau_i\cdot\nu_i = |x-M_i| & \text{if }x\in\side{P_i}{M_i}, \\[10pt]
\displaystyle \frac{|x-M_i|}{\sin\theta_{i+1}} \, \tau_{i+1}\cdot\nu_i  = - |x-M_i| & \text{if }x\in\side{M_i}{P_{i+1}}
\end{cases}
\end{equation*}
(and $X(x)\cdot\nu_j$ for $x\in\side{P_j}{P_{j+1}}$, $j\neq i$). By applying Proposition~\ref{prop:firstvar} we obtain
\begin{equation} \label{firstvar_tilting1}
\frac{\dd}{\dd t}\Big|_{t=0} \nl(\mathcal{P}_t) = 2\int_{\side{P_i}{M_i}}v_{\p}(x)|x-M_i|\dd\Hone(x)-2\int_{\side{M_i}{P_{i+1}}}v_{\p}(x)|x-M_i|\dd\Hone(x).
\end{equation}

Notice that the flow $\{\Phi_t\}_t$ does not satisfy the regularity assumption in Proposition \ref{prop:firstvar}
since there is a singularity at the point $M_i$.
We briefly describe how to deal with this issue. Take a smooth cut-off function $\varphi:[0,\infty)\to[0,1]$
with $\varphi\equiv0$ in $[0,1/2]$ and $\varphi\equiv1$ in $[1,\infty)$ and, for $\delta>0$, consider
the flow
\[
\Phi^\delta_t(x):=x+ \varphi\left(\frac{|x-M_i|^2}{\delta^2}\right)(\Phi_t(x)-x)
\]
and let $X^\delta$ be its initial velocity. Set $\mathcal{P}^\delta_t := \Phi^\delta_t(\mathcal{P})$.
Then it is easy to see that, for $t\ll 1$,
\[
\| X -X^\delta \|_{L^1(\side{P_i}{P_{i+1}})}\leq C \delta,\quad\quad\quad
| \nl(\mathcal{P}^\delta_t) - \nl(\mathcal{P}_t) | \leq C \delta t
\]
which allows us to obtain the desired result by using Proposition 3.3 for the regular flow $\{\Phi^\delta_t\}_t$
and by sending $\delta\to0$.

The first variations of the area and of the perimeter are computed in \cite[Lemma~2.6]{FraVel19} in the case of a triangle, but the proof is again the same for a general polygon, and follows from the identities
$$
|\p_t| = |\p| + o(t), \qquad \per(\p_t) = \per(\p) - \ell_i + \frac{\ell_i}{2}\biggl(\frac{\sin\theta_{i+1}-\sin t}{\sin(\theta_{i+1}+t)} + \frac{\sin\theta_i+\sin t}{\sin(\theta_i-t)}\biggr),
$$
as $t\to0$. Therefore we obtain
\begin{equation} \label{firstvar_tilting2}
\frac{\dd}{\dd t}\Big|_{t=0} |\mathcal{P}_t| = 0, \qquad \frac{\dd}{\dd t}\Big|_{t=0} \per(\p_t) = \frac{\ell_i}{2}\bigl(\psi(\theta_i)-\psi(\theta_{i+1})\bigr).
\end{equation}

Finally, the formulas in the statement follow by inserting \eqref{firstvar_sliding1}, \eqref{firstvar_sliding2}, \eqref{firstvar_tilting1}, and \eqref{firstvar_tilting2} into \eqref{proof:stationary1} and \eqref{proof:stationary2}.
\end{proof}

\subsection{Extension to nonconvex polygons} \label{subsec:nonconvex_case}

Suppose $\p\in\pol_N$ is a nonconvex polygon with $N\geq3$ vertices $P_1,\ldots,P_N$ ordered counter-clockwise. We note that Definition \ref{def:sliding} corresponding to sliding of one side works exactly as it is written in the nonconvex case, leading to the exact same first variation computations \eqref{firstvar_sliding1}-\eqref{firstvar_sliding2}. Moreover, the stationarity conditions corresponding to sliding first variation in the first part of Theorem \ref{thm:stationarity} hold verbatim in the case when $\p$ is nonconvex. 

The only changes we need to introduce are related to tilting one side of $\p$ around its midpoint. To this end, in addition to the notation presented in the Introduction, we define $\tld{\theta}_i \defeq \theta_i \mod \pi$, where $\theta_i$ denotes the (interior) angle at the vertex $P_i$. This means, $\tld{\theta}_i=\theta_i$ if $P_i$ is a convex vertex of $\p$ and $\tld{\theta}_i = \theta_i - \pi$ if $P_i$ is a concave vertex. In either case, $\tld{\theta}_i < \pi$. We modify Definition \ref{def:tilting} by explicitly defining the vertices $P_1^t,\ldots,P_N^t$ of the perturbed polygon $\p_t \in \pol_N$ as follows (see Figure \ref{fig:tilting_nonconvex}):
	\begin{equation*}
P_i^t\defeq P_i +\frac{\ell_i\sin t}{2\sin(\tld{\theta}_i-t)}\,\tld{\tau}_i \,, \qquad P_{i+1}^t\defeq P_{i+1}-\frac{\ell_i\sin t}{2\sin(\tld{\theta}_{i+1}+t)}\,\tld{\tau}_{i+1}\,,
	\end{equation*}
where
	\begin{equation*}
\tld{\tau}_i \defeq \begin{cases}
\displaystyle \frac{P_{i-1}-P_i}{|P_i-P_{i-1}|} & \text{if }\theta_i>\pi, \\[8pt]
\displaystyle \frac{P_i-P_{i-1}}{|P_i-P_{i-1}|} & \text{if }\theta_i<\pi,
\end{cases}   \qquad \tld{\tau}_{i+1} \defeq \begin{cases}
\displaystyle \frac{P_{i+2}-P_{i+1}}{|P_{i+1}-P_{i+2}|} & \text{if }\theta_{i+1}>\pi, \\[8pt]
\displaystyle \frac{P_{i+1}-P_{i+2}}{|P_{i+1}-P_{i+2}|} & \text{if }\theta_{i+1}<\pi.
\end{cases}
	\end{equation*}
	
\begin{figure}[ht]
\definecolor{qqqqff}{rgb}{0,0,1}
\definecolor{qqwuqq}{rgb}{0,0.4,0}
\begin{tikzpicture}[scale=0.8,line cap=round,line join=round]
\clip(-3.5,-4.5) rectangle (8.2,3.5);
\fill[line width=0pt,fill=qqwuqq,fill opacity=0.2] (-2,2.3) -- (-3,0) -- (-1,-3) -- (-0.4,-1.2) -- (6.64286,1.28572) -- (7,2) -- (6.5,3) -- cycle;
\draw [line width=1.2pt,dash pattern=on 3pt off 3pt,color=qqqqff] (-3,0)-- (-2,2.3);
\draw [line width=1.2pt,dash pattern=on 3pt off 3pt,color=qqqqff] (7,2)-- (6.5,3);
\draw [line width=1.2pt,color=qqqqff] (-3,0) -- (-1,-3) -- (0,0) -- (6,0) -- (7,2);
\draw [line width=1.2pt,color=qqwuqq] (-1,-3) -- (-0.4,-1.2) -- (6.64286,1.28572) -- (7,2);
\draw [line width=0.8pt,dash pattern=on 3pt off 3pt,color=qqwuqq] (0,0) -- (0.6,1.8);
\draw [line width=0.8pt,dash pattern=on 3pt off 3pt,color=qqwuqq] (6,0) -- (5,-2);
\begin{footnotesize}
\draw [fill=qqqqff] (0,0) circle (3.5pt);
\draw (-0.4,0.2) node {$P_i$};
\draw [fill=qqqqff] (6,0) circle (3.5pt);
\draw (6.5,-0.4) node {$P_{i+1}$};
\draw [fill=qqqqff] (-1,-3) circle (3.5pt);
\draw (-1,-3.5) node {$P_{i-1}$};
\draw [fill=qqqqff] (7,2) circle (3.5pt);
\draw (7.65,2) node {$P_{i+2}$};
\draw [fill] (3,0) circle (3.5pt);
\draw (3.3,-0.4) node {$M_i$};
\draw [fill=qqwuqq] (6.64286,1.28572) circle (3.5pt);
\draw (7.4,1.2) node {$P_{i+1}^t$};
\draw [fill=qqwuqq] (-0.4,-1.2) circle (3.5pt);
\draw (-0.9,-1) node {$P_i^t$};
\draw (2.2,1.9) node {$\p^t$};
\draw [->,line width=0.7pt] (4.2,0) arc (0:17:1.4);
\draw (4,0.2) node {\tiny $t$};
\draw [-,line width=0.7pt] (0.5,0) arc (0:43:1);
\draw (0.7,0.42) node {\tiny $\tld{\theta}_i$};
\end{footnotesize}
\end{tikzpicture}
\caption{A nonconvex polygon $\p$ and its variation $\p^t$ (shaded region) obtained by tilting the side $\side{P_i}{P_{i+1}}$ around its midpoint $M_i$ by an angle $t>0$.}\label{fig:tilting_nonconvex}
\end{figure}
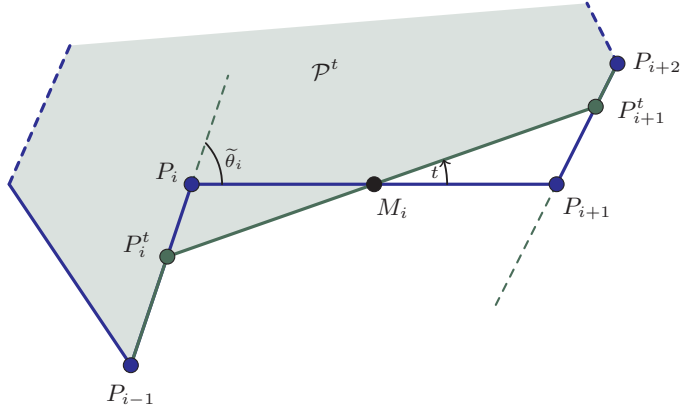
	
Using this perturbation, the calculation carried out in the second part of the proof of Theorem~\ref{thm:stationarity} applies verbatim, and we obtain the exact same stationarity conditions corresponding to tilting first variations. Therefore the conditions \eqref{eq:sliding}, \eqref{eq:tilting}, \eqref{eq:sliding2} and \eqref{eq:tilting2} hold when $\p$ is nonconvex.

\subsection{Another family of volume-preserving variations}
In this subsection we introduce a third family of area-preserving perturbations of a general polygon. We compute the corresponding first variation of the nonlocal energy and we show that it can be written as a combination of the sliding and tilting first variations.

\begin{definition} \label{def:ricIvar}
Fix three consecutive vertices $P_{i-1}$, $P_i$, $P_{i+1}$, $i\in\{1,\ldots,N\}$, of the polygon $\p$. For $t\in\R$ with $|t|$ sufficiently small, we define the polygon $\mathcal{P}_t\in\pol_N$ with vertices $P_1^t,\ldots,P_N^t$ obtained as follows (see Figure~\ref{fig:ricIvar}):
\begin{enumerate}[label = (\roman*)]
\item all vertices except $P_i$ are fixed, i.e.\ $P_j^t\defeq P_j$ for all $j\in\{1,\ldots N\}\setminus\{i\}$;
\item the vertex $P_i^t$ is given by
$$
P_i^t = P_i + t\,\frac{P_{i+1}-P_{i-1}}{|P_{i+1}-P_{i-1}|},
$$
that is, $P_i^t$ lies on the line through $P_i$ parallel to the diagonal $\side{P_{i-1}}{P_{i+1}}$, at a distance $|t|$ from $P_i$.
\end{enumerate}
\end{definition}

\begin{figure}[ht]
\definecolor{qqqqff}{rgb}{0,0,1}
\definecolor{qqwuqq}{rgb}{0,0.4,0}
\begin{tikzpicture}[scale=0.7,line cap=round,line join=round]
\clip(-6,-3) rectangle (6,5);
\fill[line width=0pt,fill=qqwuqq,fill opacity=0.2] (-2.622,-2.756) -- (-4,0) -- (-2,3.96) -- (4,0) -- (4.54,-2.76) -- cycle;
\draw [line width=1.4pt,color=qqqqff] (-4,0)-- (-3.37,3.96);
\draw [line width=1.4pt,color=qqqqff] (-3.37,3.96)-- (4,0);
\draw [line width=1.4pt,color=qqqqff] (4,0)-- (4.39,-2);
\draw [line width=1.4pt,color=qqqqff] (-4,0)-- (-3,-2);
\draw [line width=1.4pt,dash pattern=on 3pt off 3pt,color=qqqqff] (-3,-2)-- (-2.622,-2.756);
\draw [line width=1.4pt,dash pattern=on 3pt off 3pt,color=qqqqff] (4.39,-2)-- (4.54,-2.76);
\draw [line width=1.4pt,dash pattern=on 3pt off 3pt,color=qqwuqq] (-4,0)-- (-2,3.96);
\draw [line width=1.4pt,dash pattern=on 3pt off 3pt,color=qqwuqq] (-2,3.96)-- (4,0);
\draw [line width=1pt,dash pattern=on 3pt off 3pt] (-4,0)-- (4,0);
\begin{footnotesize}
\draw [fill=qqqqff] (-4,0) circle (3pt);
\draw (-4.5,-0.2) node {$P_{i-1}$};
\draw [fill=qqqqff] (4,0) circle (3pt);
\draw (4.6,-0.2) node {$P_{i+1}$};
\draw [fill=qqqqff] (-3.37,3.96) circle (3pt);
\draw (-3.5,4.25) node {$P_i$};
\draw [fill=qqwuqq] (-2,3.96) circle (3pt);
\draw (-1.7,4.25) node {$P_i^t$};
\draw (-2.6,4.2) node {$t$};
\end{footnotesize}
\draw [->,line width=1pt] (-3.37,3.96)-- (-2,3.96);
\draw [-,line width=0.7pt] (-3.3,0) arc (0:80:0.7);
\draw (-3,0.5) node {\tiny $\alpha_{i}^-$};
\draw [-,line width=0.7pt] (3.1,0) arc (180:150:0.9);
\draw (2.7,0.25) node {\tiny $\alpha_{i}^+$};
\end{tikzpicture}
\caption{A polygon $\p$ and its variation $\p^t$ (shaded region) as in Definition~\ref{def:ricIvar}, obtained by moving the vertex $P_i$ parallel to the diagonal $\side{P_{i-1}}{P_{i+1}}$ at a distance $t>0$.}\label{fig:ricIvar}
\end{figure}
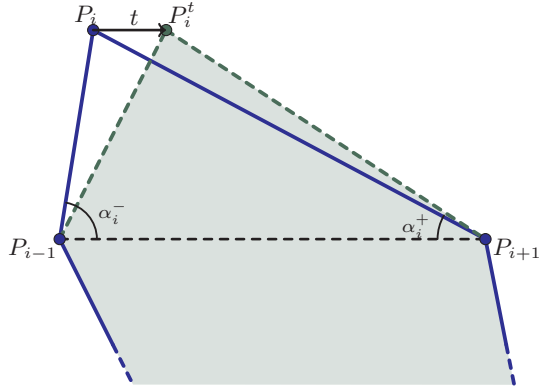

In the rest of this subsection we work under the assumption that the angle $\theta_i$ at the vertex $P_i$ is smaller than $\pi$, as we will use the variation in Definition \ref{def:ricIvar} only for a convex vertex.

\begin{proposition} \label{prop:ricIvar}
Let $\mathcal{P}\in\pol_N$ and let $\{\mathcal{P}_t\}_t$ be the family of perturbations of $\mathcal{P}$ as in Definition~\ref{def:ricIvar}, obtained by moving the vertex $P_i$ parallel to the diagonal $\side{P_{i-1}}{P_{i+1}}$. Then
\begin{equation} \label{firstvar_ric}
\begin{split}
I_i \defeq \frac{\dd}{\dd t}\Big|_{t=0} \nl(\mathcal{P}_t)
& = \frac{2\sin \alpha_{i}^+}{\ell_{i}}\int_{\side{P_{i}}{P_{i+1}}}v_{\p}(x)|x-P_{i+1}|\dd\Hone(x) \\
& \qquad - \frac{2\sin\alpha_{i}^-}{\ell_{i-1}}\int_{\side{P_{i-1}}{P_i}}v_{\p}(x)|x-P_{i-1}|\dd\Hone(x),
\end{split}
\end{equation}
\begin{equation} \label{firstvar_ric2}
\frac{\dd}{\dd t}\Big|_{t=0} |\mathcal{P}_t| = 0, \qquad \frac{\dd}{\dd t}\Big|_{t=0} \per({\p}_t) = \cos\alpha_{i}^--\cos\alpha_{i}^+.
\end{equation}
where $\alpha_{i}^-\in(0,\pi)$ is the angle between $\side{P_{i-1}}{P_{i+1}}$ and $\side{P_{i-1}}{P_i}$, and $\alpha_{i}^+\in(0,\pi)$ is the angle between $\side{P_{i-1}}{P_{i+1}}$ and $\side{P_i}{P_{i+1}}$.
\end{proposition}

\begin{proof}
A flow $\{\Phi_t\}_t$ which induces the perturbation $\{\p_t\}_t$ is explicitly given by
\begin{equation} \label{eq:ricflow}
\Phi_t(x) =
\begin{cases}
\displaystyle x + t \, \frac{|x-P_{i-1}|}{\ell_{i-1}} \, \tau & \text{if }x\in\side{P_{i-1}}{P_i}, \\[10pt]
\displaystyle x + t \, \frac{|x-P_{i+1}|}{\ell_{i}} \, \tau& \text{if }x\in\side{P_i}{P_{i+1}},
\end{cases}
\quad\qquad\text{where}\quad\tau\defeq\frac{P_{i+1}-P_{i-1}}{|P_{i+1}-P_{i-1}|}\,.
\end{equation}
Then the normal component of the initial velocity is
\begin{align*}
X\cdot\nu_{i-1} = -\frac{\sin\alpha_{i}^-}{\ell_{i-1}}|x-P_{i-1}| & \qquad\text{on }\side{P_{i-1}}{P_i},\\
X\cdot\nu_i = \frac{\sin\alpha_{i}^+}{\ell_i}|x-P_{i+1}| & \qquad\text{on }\side{P_{i}}{P_{i+1}},
\end{align*}
and we obtain \eqref{firstvar_ric} by applying Proposition~\ref{prop:firstvar}.

The first condition in \eqref{firstvar_ric2} follows from the fact that this perturbation is area preserving: $|\p_t|=|\p|$. The formula for the first variation of the perimeter follows from the identity
\begin{align*}
\per(\p_t)
& = \per(\p) + \sqrt{\ell_{i-1}^2+2t\ell_{i-1}\cos\alpha_i^-+t^2} - \ell_{i-1} + \sqrt{\ell_i^2 -2t\ell_i\cos\alpha_i^++t^2}-\ell_i \\
& = \per(\p) + t (\cos\alpha_{i}^--\cos\alpha_{i}^+) + o(t),
\end{align*}
which can be checked by elementary geometric arguments.
\end{proof}

Arguing as in Theorem~\ref{thm:stationarity}, we find that the stationarity conditions of a polygon $\p$ with respect to the variation in Definition~\ref{def:ricIvar} are
\begin{equation} \label{stationarity-ric1}
I_i = 0 \qquad\text{(under area constraint)}
\end{equation}
and
\begin{equation} \label{stationarity-ric2}
I_i = 2\bar{\sigma}\bigl(\cos\alpha_{i}^--\cos\alpha_{i}^+\bigr) \qquad\text{(under perimeter constraint),}
\end{equation}
where $\bar{\sigma}=\frac{\sigma}{\per(\p)}$. In the next proposition we show that these conditions follow by the stationarity conditions with respect to the sliding and tilting perturbations.

\begin{proposition} \label{prop:ricIvar2}
If the polygon $\p$ satisfies \eqref{eq:sliding} and \eqref{eq:tilting} on the two consecutive sides $\side{P_{i-1}}{P_i}$ and $\side{P_i}{P_{i+1}}$, then \eqref{stationarity-ric1} holds. If $\p$ satisfies \eqref{eq:sliding2} and \eqref{eq:tilting2} on the two consecutive sides $\side{P_{i-1}}{P_i}$ and $\side{P_i}{P_{i+1}}$, then \eqref{stationarity-ric2} holds.
\end{proposition}

\begin{proof}
We have
\begin{align}\label{proof:ricIvar}
I_i & = \frac{2\sin\alpha_{i}^+}{\ell_{i}} \biggl[ \int_{\side{P_{i}}{M_{i}}}v_{\p}(x) \Bigl(|x-M_{i}|+\frac{\ell_i}{2}\Bigr) + \int_{\side{M_{i}}{P_{i+1}}}v_{\p}(x) \Bigl( \frac{\ell_i}{2} - |x-M_{i}| \Bigr) \biggr] \nonumber\\
& \hskip.2cm - \frac{2\sin\alpha_{i}^-}{\ell_{i-1}} \biggl[ \int_{\side{P_{i-1}}{M_{i-1}}}v_{\p}(x) \Bigl(\frac{\ell_{i-1}}{2}-|x-M_{i-1}|\Bigr) + \int_{\side{M_{i-1}}{P_i}}v_{\p}(x) \Bigl(|x-M_{i-1}|+\frac{\ell_{i-1}}{2}\Bigr) \biggr] \nonumber \\
& = \sin\alpha_{i}^+ \int_{\side{P_{i}}{P_{i+1}}}v_{\p}(x)\dd\Hone(x) - \sin\alpha_{i}^- \int_{\side{P_{i-1}}{P_{i}}}v_{\p}(x)\dd\Hone(x) \nonumber\\
& \hskip.2cm + \frac{2\sin\alpha_{i}^+}{\ell_{i}} \biggl[ \int_{\side{P_{i}}{M_{i}}}v_{\p}(x) |x-M_{i}|\dd\Hone(x) - \int_{\side{M_{i}}{P_{i+1}}}v_{\p}(x) |x-M_{i}|\dd\Hone(x) \biggr] \nonumber\\
& \hskip.2cm + \frac{2\sin\alpha_{i}^-}{\ell_{i-1}} \biggl[ \int_{\side{P_{i-1}}{M_{i-1}}}v_{\p}(x) |x-M_{i-1}|\dd\Hone(x) - \int_{\side{M_{i-1}}{P_i}}v_{\p}(x) |x-M_{i-1}|\dd\Hone(x) \biggr].
\end{align}
Assume now that \eqref{eq:sliding} and \eqref{eq:tilting} hold on the two consecutive sides $\side{P_{i-1}}{P_i}$ and $\side{P_i}{P_{i+1}}$. Then by \eqref{eq:tilting} the last two lines in \eqref{proof:ricIvar} are equal to zero, therefore
\begin{align*}
I_i & = \sin\alpha_{i}^+ \int_{\side{P_{i}}{P_{i+1}}}v_{\p}(x)\dd\Hone(x) - \sin\alpha_{i}^- \int_{\side{P_{i-1}}{P_{i}}}v_{\p}(x)\dd\Hone(x) \\
& = \ell_i\sin\alpha_{i}^+ \biggl( \frac{1}{\ell_i}\int_{\side{P_i}{P_{i+1}}}v_{\p}(x)\dd\Hone(x) - \frac{1}{\ell_{i-1}}\int_{\side{P_{i-1}}{P_i}}v_{\p}(x)\dd\Hone(x) \biggr)
\xupref{eq:sliding}{=}0,
\end{align*}
where we used the identity $\frac{\ell_{i-1}}{\sin\alpha_{i}^+} = \frac{\ell_i}{\sin\alpha_{i}^-}$. This proves \eqref{stationarity-ric1}.

Assume instead \eqref{eq:sliding2} and \eqref{eq:tilting2} hold on the two consecutive sides $\side{P_{i-1}}{P_i}$ and $\side{P_i}{P_{i+1}}$. Then substituting in \eqref{proof:ricIvar} we obtain
\begin{align*}
I_i &= \bar{\sigma}\sin\alpha_{i}^+\bigl(\psi(\theta_i)+\psi(\theta_{i+1})\bigr) - \bar{\sigma}\sin\alpha_{i}^-\bigl(\psi(\theta_{i-1})+\psi(\theta_{i})\bigr) \\
& \qquad + \bar{\sigma}\sin\alpha_{i}^+\bigl(\psi(\theta_i)-\psi(\theta_{i+1})\bigr) + \bar{\sigma}\sin\alpha_{i}^-\bigl(\psi(\theta_{i-1})-\psi(\theta_{i})\bigr) \\
& = 2\bar{\sigma}\psi(\theta_i)\bigl(\sin\alpha_{i}^+ - \sin\alpha_{i}^-\bigr) \\
& \xupref{eq:cotangent}{=} \frac{2\bar{\sigma}}{\sin\theta_i}\Bigl( \cos\theta_i\sin\alpha_{i}^+ + \sin\alpha_{i}^+ - \cos\theta_i\sin\alpha_{i}^- - \sin\alpha_{i}^-\Bigr) \\
& = 2\bar{\sigma} \bigl(\cos\alpha_{i}^--\cos\alpha_{i}^+\bigr),
\end{align*}
where the last equality follows by elementary trigonometric relations, using the fact that $\alpha_{i}^-+\theta_i+\alpha_{i}^+=\pi$. This proves \eqref{stationarity-ric2}.
\end{proof}


The strategy to prove Theorem~\ref{thm2} for quadrilaterals is mainly based on the previous proposition: indeed, we will prove that if two consecutive sides have different lengths $\ell_{i-1}\neq\ell_i$, then the first variation of the nonlocal energy with respect to the perturbation in Definition~\ref{def:ricIvar} is different from zero; in turn, by Proposition~\ref{prop:ricIvar2} the quadrilateral does not satisfy the stationarity conditions \eqref{eq:sliding} and \eqref{eq:tilting}. As a consequence, a quadrilateral satisfying both \eqref{eq:sliding} and \eqref{eq:tilting} is necessarily equilateral (that is, it is a rhombus). In a final step we will also show that the polygon must be equiangular. The proof of Theorem~\ref{thm3} follows by a similar argument, comparing the sign of $I_i$ with the sign of the right-hand side of \eqref{stationarity-ric2}.


\section{Overdetermined problem for triangles}\label{sec:triangles}

In this section we prove Theorem~\ref{thm2} and Theorem~\ref{thm3} for triangles, i.e., in the case $N=3$. We give two alternative proofs of Theorem~\ref{thm2}; the ideas used in both proofs will appear in the next section when we prove the results for quadrilaterals.

\begin{remark} \label{rmk:triangles}
Notice that, in the proof of Theorem~\ref{thm2} for triangles, we will use only condition \eqref{eq:tilting}. In fact, \eqref{eq:sliding} is satisfied by \emph{every} triangle, since the operation of sliding one side and rescaling to restore the area leaves the triangle unchanged, and the corresponding first variation is equal to zero. For the same reason, also the condition \eqref{eq:sliding2} is satisfied by every triangle.
\end{remark}

\subsection{Equilateral triangles via reflection arguments} \label{subsec:triangles_reflection}
In the first proof we will use reflection arguments to obtain that if $\p\in\pol_3$ satisfies \eqref{eq:tilting} or \eqref{eq:tilting2}, then $\p$ is equilateral.

\begin{proof}[First proof of Theorem~\ref{thm2} in the case $N=3$] Let $\p \in \pol_3$ be an arbitrary triangle and assume that $\p$ satisfies the condition \eqref{eq:tilting}. Without loss of generality, assume that $\p$ is translated and rotated so that the midpoint $M_1$ of the side $\side{P_1}{P_2}$ coincides with the origin of the $(x_1,x_2)$-plane, and the side $\side{P_1}{P_2}$ lies on the $x_1$-axis.

We show that, assuming condition \eqref{eq:tilting} holds on the side $\side{P_1}{P_2}$, we have $\theta_1=\theta_2$. Suppose by contradiction $\theta_1< \theta_2$. Let $\tld{\p}$ denote the reflection of $\p$ with respect to the $x_2$-axis, and define the sets $D \defeq \p \setminus \tld{\p}$ and $\tld{D} \defeq \tld{\p} \setminus \p$ (see Figure \ref{fig:triangles1}).

Let $x\in \side{M_1}{P_2}$ and denote by $\tld{x} \in \side{P_1}{M_1}$ the reflection of $x$ in the $x_2$-axis. Then
	\begin{equation} \label{eq:reflection_calc}
		\begin{aligned}
			v_{\tld{\p}}(x)  - v_{\p}(x) &= \int_{\tld{\p}} K(|x-y|) \dd y - \int_{\p} K(|x-y|) \dd y \\
														&= \int_{\tld{D}} K(|x-y|) \dd y - \int_{D} K(|x-y|) \dd y \\
												       &= \int_D \Big( K(|\tld{x}-y|) - K(|x-y|) \Big) \dd y < 0,
		\end{aligned}
	\end{equation}
since $|\tld{x}-y| > |x-y|$ for all $y\in D$ and $K$ is strictly decreasing. This implies that $v_{\tld{\p}}(x)  < v_{\p}(x)$ for all $x\in \side{M_1}{P_2}$. Multiplying both sides by $|x- M_1|$ and integrating, then, yields
	\[
		\int_{\side{M_1}{P_2}} v_\p(x)|x-M_1| \dd \Hone(x) > \int_{\side{M_1}{P_2}} v_{\tld{\p}}(x)|x-M_1| \dd \Hone(x) = \int_{\side{P_1}{M_1}} v_\p(x)|x- M_1|\dd \Hone(x),
	\]
which contradicts the condition \eqref{eq:tilting} on $\side{P_1}{P_2}$. This implies that $\theta_1= \theta_2$, i.e., $\p$ is isosceles.

Repeating the argument for another pair of angles, say $\theta_2$ and $\theta_3$, we obtain that $\theta_1=\theta_2=\theta_3$, i.e., $\p$ is equilateral.
\end{proof}

\begin{proof}[Proof of Theorem~\ref{thm3} in the case $N=3$]
As in the previous proof, assuming condition \eqref{eq:tilting2} on the side $\side{P_1}{P_2}$, we show that $\theta_1=\theta_2$. Suppose by contradiction that $\theta_1<\theta_2$. Then, as before, we obtain that the left-hand side of \eqref{eq:tilting2} (for $i=1$) is strictly negative, and therefore $\psi(\theta_1)-\psi(\theta_2)<0$. This is a contradiction since $\psi$ is monotone decreasing and $\theta_1<\theta_2$.
\end{proof}

Notice that, by the previous proofs, it is sufficient to assume that \eqref{eq:tilting} or \eqref{eq:tilting2} hold just for two of the three sides of a triangle in order to deduce that it is equilateral.

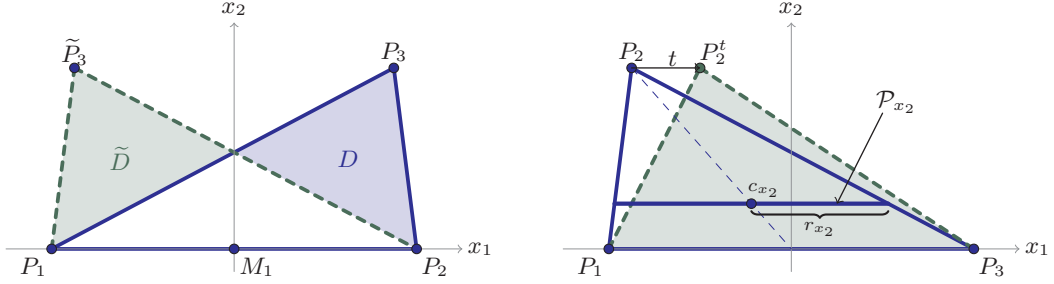
\begin{figure}[ht]
\definecolor{qqqqff}{rgb}{0,0,1}
\definecolor{qqwuqq}{rgb}{0,0.4,0}
\begin{tikzpicture}[scale=0.6,line cap=round,line join=round]
\clip(-6,-1) rectangle (6,7);
\fill[line width=0pt,fill=qqwuqq,fill opacity=0.2] (-3.5,4) -- (0,32/15) -- (-4,0) -- cycle;
\fill[line width=0pt,fill=qqqqff,fill opacity=0.2] (3.5,4) -- (0,32/15) -- (4,0) -- cycle;
\draw [line width=1.4pt,color=qqqqff] (-4,0)-- (3.5,4);
\draw [line width=1.4pt,color=qqqqff] (3.5,4)-- (4,0);
\draw [line width=1.4pt,color=qqqqff] (4,0)-- (-4,0);
\draw [line width=1.4pt,dash pattern=on 3pt off 3pt,color=qqwuqq] (-4,0)-- (-3.5,4);
\draw [line width=1.4pt,dash pattern=on 3pt off 3pt,color=qqwuqq] (-3.5,4)-- (4,0);
\draw [->,line width=0.3pt,color=black!50] (-5,0)-- (5,0);
\draw [->,line width=0.3pt,color=black!50] (0,-0.5)-- (0,5);
\begin{footnotesize}
\draw [fill=qqqqff] (-4,0) circle (3pt);
\draw (-4.4,-0.4) node {$P_1$};
\draw [fill=qqqqff] (4,0) circle (3pt);
\draw (4.4,-0.4) node {$P_2$};
\draw (5.4,0) node {$x_1$};
\draw (0,5.3) node {$x_2$};
\draw [fill=qqqqff] (3.5,4) circle (3pt);
\draw (3.5,4.35) node {$P_3$};
\draw [fill=qqqqff] (-3.5,4) circle (3pt);
\draw (-3.5,4.35) node {$\tld{P}_3$};
\draw [fill=qqqqff] (0,0) circle (3pt);
\draw (0.5,-0.4) node {$M_1$};
\draw [color=qqwuqq] (-2.5,2) node {$\tld{D}$};
\draw [color=qqqqff] (2.5,2) node {$D$};
\end{footnotesize}
\end{tikzpicture}
\begin{tikzpicture}[scale=0.6,line cap=round,line join=round]
\clip(-6,-1) rectangle (6,7);
\fill[line width=0pt,fill=qqwuqq,fill opacity=0.2] (-4,0) -- (-2,4) -- (4,0) -- cycle;
\draw [line width=1.4pt,color=qqqqff] (-4,0)-- (-3.5,4);
\draw [line width=1.4pt,color=qqqqff] (-3.5,4)-- (4,0);
\draw [line width=1.4pt,color=qqqqff] (4,0)-- (-4,0);
\draw [line width=1.4pt,dash pattern=on 3pt off 3pt,color=qqwuqq] (-4,0)-- (-2,4);
\draw [line width=1.4pt,dash pattern=on 3pt off 3pt,color=qqwuqq] (-2,4)-- (4,0);
\draw [->,line width=0.3pt,color=black!50] (-5,0)-- (5,0);
\draw [->,line width=0.3pt,color=black!50] (0,-0.5)-- (0,5);
\draw [line width=0.3pt,dash pattern=on 3pt off 3pt,color=qqqqff] (-3.5,4) -- (0,0);
\draw [line width=1.6pt,color=qqqqff] (-31/8,1) -- (17/8,1);
\begin{footnotesize}
\draw [fill=qqqqff] (-4,0) circle (3pt);
\draw (-4.4,-0.4) node {$P_1$};
\draw [fill=qqqqff] (4,0) circle (3pt);
\draw (4.4,-0.4) node {$P_3$};
\draw (5.4,0) node {$x_1$};
\draw (0,5.3) node {$x_2$};
\draw [fill=qqqqff] (-3.5,4) circle (3pt);
\draw (-3.5,4.35) node {$P_2$};
\draw [fill=qqqqff] (-7/8,1) circle (3pt);
\draw (-0.6,1.3) node {{\tiny $c_{x_2}$}};
\draw [fill=qqwuqq] (-2,4) circle (3pt);
\draw (-1.7,4.35) node {$P_2^t$};
\draw (-2.6,4.2) node {$t$};
\draw (2.3,3.25) node {$\p_{x_2}$};
\draw [
    thick,
    decoration={
        brace,
        mirror,
        raise=0.1cm
    },
    decorate
] (-7/8,1) -- (17/8,1) 
node [pos=0.5,anchor=north,yshift=-0.12cm] {{\tiny $r_{x_2}$}}; 
\draw [->,line width=0.5pt] (2,3) -- (1,1);
\end{footnotesize}
\draw [->,line width=0.5pt] (-3.5,4)-- (-2,4);
\end{tikzpicture}
\caption{The sets $D=\p\setminus\tld{\p}$ and $\tld{D}=\tld{\p}\setminus\p$ used in the reflection argument (left) and the area-preserving variation $\Phi_t(\p)$ used in the first variation argument (right).}\label{fig:triangles1}
\end{figure}

\subsection{Equilateral triangles via first variation arguments} \label{subsec:triangles_first_var}
Our second proof is inspired by the arguments in \cite[Section~2.2.2]{CaHiVoYa2016} where the authors study the interaction energy $\nl$ under continuous Steiner symmetrizations. Instead, we will use the the volume-preserving variations in Definition \ref{def:ricIvar} and show that the first variation of $\nl$ along these perturbations is strictly positive unless the triangle is isosceles.

The main idea of the proof is to express the first variation of the energy using slices of the triangle and to compute the derivative of the interaction between two slices. Let us start by fixing our notation for this subsection. Let $\p\in\pol_3$ and, as in Figure~\ref{fig:triangles1}, fix the coordinate axes so that the midpoint $M_3$ of the side $\side{P_3}{P_{1}}$ coincides with the origin of the $(x_1,x_2)$-plane, the side $\side{P_1}{P_3}$ lies on the $x_1$-axis, the point $P_1$ is on the negative $x_1$-axis, and the point $P_2$ is in the upper half-plane. Assume $\theta_1 > \theta_3$. For $x_2>0$, let
\begin{equation*}
\p_{x_2} \defeq \big\{ x_1 \in \R \colon (x_1,x_2)\in\p \big\}.
\end{equation*}
Note that $\p_{x_2} \subset \R$ is an interval $(c_{x_2}-r_{x_2},c_{x_2}+r_{x_2})$ with $r_{x_2}\geq 0$ and $c_{x_2}<0$ since $\theta_1>\theta_3$.

Let $\{\Phi_t\}_t$ denote the flow \eqref{eq:ricflow} introduced in the proof of Proposition~\ref{prop:ricIvar}. Then
	\[
		\Phi_t(\p_{x_2}) = \p_{x_2} + \alpha x_2\, t \qquad \text{ with } \quad \alpha \defeq \frac{1}{\ell_1 \sin \theta_1},
	\]
and
	\[
		\Phi_t(\p) = \big\{ (x_1,x_2) \colon x_1 \in \Phi_t(\p_{x_2}) \big\}.
	\]
The next lemma shows that the derivative of the interaction between two slices is strictly positive for $C^1$ and even interaction kernels.

\begin{lemma} \label{lem:carrillo}
Let $W \in C^1(\R)$ be an even function with $W^\pr(r)<0$ for $r > 0$. Then for every $x_2,\, y_2>0$ setting
	\[
		\calI_W\big[\p_{x_2},\p_{y_2}\big](t) \defeq \int_\R \! \int_\R W(x_1-y_1)\chi_{\Phi_t(\p_{x_2})}(x_1)\chi_{\Phi_t(\p_{y_2})}(y_1) \dd x_1 \dd y_1
	\]
we have
	\[
		\frac{\dd}{\dd t}\Big|_{t=0} \calI_W\big[\p_{x_2},\p_{y_2}\big](t) \geq C_W\, \alpha \min\{r_{x_2},r_{y_2}\} |c_{x_2}-c_{y_2}| \, |x_2-y_2|,
	\]
where
	\begin{equation}\label{eq:C_W}
		C_W = \min \big\{ |W^\pr(r)| \colon r\in \big[|c_{x_2}-c_{y_2}|/2, |c_{x_2}-c_{y_2}|+r_{x_2}+r_{y_2}\big] \big\}.
	\end{equation}
\end{lemma}

\begin{proof}
For simplicity of notation, we will drop the subscripts on the centers and radii of $\p_{x_2}$ and $\p_{y_2}$. Namely, let $\p_{x_2} = (c_x - r_x,c_x+r_x)$ and $\p_{y_2}=(c_y-r_y,c_y+r_y)$. Assume, without loss of generality, that $y_2>x_2>0$. Then $r_y < r_x$ and $c_y<c_x<0$, since $\theta_1 > \theta_3$.

Then 
	\[
	\calI_W\big[\p_{x_2},\p_{y_2}\big](t) = \int_{-r_x}^{r_x}\!\int_{-r_y}^{r_y} W\big(x_1-y_1 + c_x-c_y + \alpha x_2\, t - \alpha y_2 \, t\big) \dd y_1 \dd x_1,
	\]
and we get that
\begin{align*}
\frac{\dd}{\dd t}\Big|_{t=0} \calI_W\big[\p_{x_2},\p_{y_2}\big](t)
& = \alpha(x_2-y_2) \int_{-r_x}^{r_x}\!\int_{-r_y}^{r_y} W^\pr\big((x_1-y_1) + (c_x-c_y)\big) \dd y_1 \dd x_1 \\
& = \alpha(x_2-y_2) \iint_{R} W^\pr(x_1-y_1)\dd x_1 \dd y_1,
\end{align*}
where $R$ denotes the rectangle $\big[-r_x + (c_x-c_y),r_x+(c_x-c_y)\big] \times \big[-r_y,r_y \big]$ in the $(x_1,y_1)$-plane. Now, let 
	\begin{gather*}
		R^+\defeq R \cap \{y_1>x_1\}, \ R^- \defeq R \cap \{y_1 < x_1\}, \ \tld{R}^- \defeq R^- \cap \{x_1 \leq r_x\}, \\
		\text{ and } D \defeq  \big[r_x + (c_x-c_y)/2,r_x+(c_x-c_y)\big] \times \big[-r_y,r_y \big],
	\end{gather*}
and note that $R^+$, $\tld{R}^-$, $D$ are disjoint subsets of $R$, with $D\subset R^-$ (see Figure \ref{fig:rectangle}). Moreover $W^\pr(x_1-y_1) > 0$ on $R^+$ and $W^{\pr}(x_1-y_1)<0$ on $R^-$. Since $x_2-y_2<0$ we get that
	\beqn \label{eq:carrillo_est1}
		\begin{aligned}
			\frac{\dd}{\dd t}\Big|_{t=0} \calI_W\big[\p_{x_2},\p_{y_2}\big](t) &= \alpha(x_2-y_2) \iint_{R} W^\pr(x_1-y_1)\dd x_1 \dd y_1 \\
																														  &\geq \alpha(x_2-y_2) \left[ \iint_{R^+} W^\pr(x_1-y_1)\dd x_1 \dd y_1 \right. \\ & \ \quad + \left. \iint_{\tld{R}^{-}} W^\pr(x_1-y_1) \dd x_1 \dd y_1
																														  + \iint_D W^\pr(x_1-y_1)\dd x_1 \dd y_1 \right].
		\end{aligned}
	\eeqn
	
\begin{figure}[ht]  
\definecolor{qqqqff}{rgb}{0,0,1}
\definecolor{qqwuqq}{rgb}{0,0.4,0}
\begin{tikzpicture}[scale=0.5,line cap=round,line join=round]
\clip(-8,-5) rectangle (10.5,5);
\draw [->,line width=0.3pt,color=black!50] (-4,0)-- (9,0);
\draw [->,line width=0.3pt,color=black!50] (0,-5)-- (0,5);
\draw [line width=0.3pt,color=black!50] (-4,-4)-- (5,5);
\draw [line width=1.4pt,color=qqqqff] (-3,-4) -- (7,-4) -- (7,4) -- (-3,4) -- (-3,-4);
\draw [line width=0.3pt,color=qqwuqq] (6,4)-- (6,-4);
\draw [line width=0.3pt,color=black!50] (5,4)-- (5,-4);
\fill[line width=0pt,fill=qqwuqq,fill opacity=0.2] (6,4) -- (7,4) -- (7,-4) -- (6,-4) -- cycle;
\fill[line width=0pt,fill=qqqqff,fill opacity=0.2] (-3,-3) -- (4,4) -- (-3,4) -- cycle;
\fill[line width=0pt,fill=gray,fill opacity=0.2] (4,4) -- (5,4) -- (5,-4) -- (-3,-4) -- (-3,-3) -- cycle;

\draw [fill=qqqqff] (-3,0) circle (2pt);
\draw [fill=qqqqff] (0,4) circle (2pt);
\draw [fill=qqqqff] (0,-4) circle (2pt);
\draw [fill=qqqqff] (7,0) circle (2pt);
\draw [fill=black!50] (5,0) circle (2pt);
\begin{tiny}
\draw (8.7,-0.3) node {$x_1$};
\draw (0.4,4.7) node {$y_1$};
\draw (5.9,4.7) node {$y_1=x_1$};
\draw (8.7,0.3) node {$r_x+c_x-c_y$};
\draw (-4.8,0.3) node {$-r_x+c_x-c_y$};
\draw (4.6,0.3) node {$r_x$};
\draw (0.4,3.65) node {$r_y$};
\draw (-0.48,-3.7) node {$-r_y$};
\end{tiny}
\begin{footnotesize}
\draw [color=qqwuqq] (6.5,-3.55) node {$D$};
\draw [color=qqqqff] (-2.35,3.55) node {$R^+$};
\draw [color=black!80] (4.5,-3.45) node {$\tld{R}^-$};
\end{footnotesize}
\end{tikzpicture}
\caption{Subsets $R^+$, $\tld{R}^-$ and $D$ of the rectangle $R$.}\label{fig:rectangle}
	\end{figure}
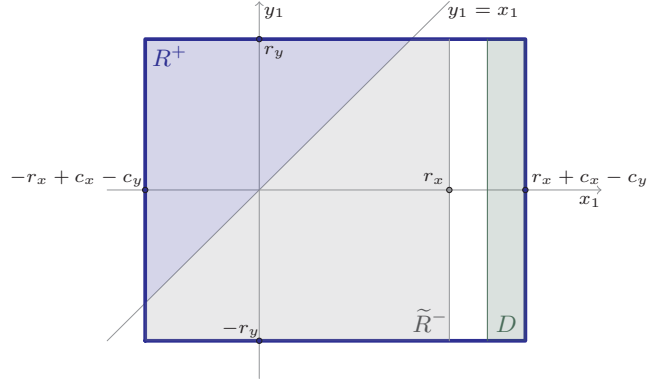
	
Since $R^+ \cup \tld{R}^-$ is a rectangle with center $(\frac{c_x-c_y}{2},0)$, for every $h>0$ we have that $\mathcal{L}^1(R^+ \cap \{y_1 = x_1 + h\}) \leq \mathcal{L}^1(\tld{R}^- \cap \{y_1 = x_1 - h\})$. This, in turn, implies that
	\[
		\iint_{R^+} W^\pr(x_1-y_1)\dd x_1 \dd y_1 + \iint_{\tld{R}^{-}} W^\pr(x_1-y_1) \dd x_1 \dd y_1 \leq 0.
	\]
Returning to \eqref{eq:carrillo_est1} and using the fact that $W$ is even, we obtain
	\begin{align*}
		\frac{\dd}{\dd t}\Big|_{t=0} \calI_W\big[\p_{x_2},\p_{y_2}\big](t) & \geq \alpha (y_2-x_2) \iint_D W^\pr(y_1-x_1)\dd x_1 \dd y_1 \\
																													  & \geq \alpha (y_2 - x_2)|D| \min_{(x_1,y_1)\in D} W^\pr(y_1-x_1) \\
																													  &= \alpha(y_2-x_2) r_y (c_x-c_y) \min_{(x_1,y_1)\in D} W^\pr(y_1-x_1) \\
																													  &\geq C_W \, \alpha (y_2-x_2) r_y(c_x-c_y),
	\end{align*}
which yields the result.
\end{proof}

Now we are ready to give the second proof of Theorem~\ref{thm2} in the case $N=3$.

\begin{proof}[Second proof of Theorem~\ref{thm2} in the case $N=3$] Let $\p\in\pol_3$ be as above and assume by contradiction that $\theta_1 > \theta_3$. As in the proof of Proposition~\ref{prop:firstvar}, to avoid problems in differentiating we regularize the kernel by introducing a small parameter $\delta>0$ and obtain $K_\delta$ and $\nl_\delta$ given by \eqref{eq:Kdelta}. Set $K_{\delta,l}(r) \defeq K_\delta(\sqrt{l^2+r^2}) = K(\sqrt{l^2+r^2}+\delta)$. Note that $K_{\delta,l}$ satisfies the assumptions of Lemma~\ref{lem:carrillo}, namely it is a $C^1$, even function with $K_{\delta,l}'(r)<0$ for $r>0$. Then, by Fubini's theorem,
	\[
		\nl_\delta(\Phi_t(\p)) = \int_{\R^2} \! \int_{\R^2} K_\delta(|x-y|)\chi_{\Phi_t(\p)}(x)\chi_{\Phi_t(\p)}(y)\dd x \dd y = \int_{\R}\!\int_{\R} \calI_{K_{\delta,l}}\big[\p_{x_2},\p_{y_2}\big](t) \dd x_2 \dd y_2
	\]
with $l \defeq |x_2-y_2|$.
Hence, by Lemma~\ref{lem:carrillo}, we get
\begin{align*}
\frac{\dd}{\dd t} \Big|_{t=0} \nl_\delta(\Phi_t(\p))
& = \int_{\R} \! \int_{\R} \frac{\dd}{\dd t}\Big|_{t=0} \calI_{K_{\delta,l}}\big[\p_{x_2},\p_{y_2}\big](t) \dd x_2 \dd y_2 \\
& \geq \int_{\R}\int_{\R}C_{K_{\delta,l}} \, \alpha \min\{r_{x_2},r_{y_2}\} |c_{x_2}-c_{y_2}| \, |x_2-y_2| \dd x_2\dd y_2 \geq C_\delta
\end{align*}
for some constant $C_\delta>0$, where $C_{K_{\delta,l}}$ is given by \eqref{eq:C_W}. Since $C_{\delta}$ is bounded away from zero uniformly in $\delta$, as in the proof of Proposition~\ref{prop:firstvar} we can pass to the limit as $\delta \to 0$ and obtain
	\[
		\frac{\dd}{\dd t} \Big|_{t=0} \nl(\Phi_t(\p)) = \lim_{\delta\to0}\frac{\dd}{\dd t} \Big|_{t=0} \nl_\delta(\Phi_t(\p))  > 0.
	\]
However, due to Proposition~\ref{prop:ricIvar2}, this contradicts the fact that $\p$ satisfies the condition \eqref{eq:tilting}. Therefore $\theta_1 = \theta_3$, i.e., $\p$ is isosceles. Repeating this argument for all pairs of angles, we get that $\theta_1=\theta_2=\theta_3$; hence, $\p$ is equilateral.
\end{proof}


\section{Overdetermined problem for quadrilaterals}\label{sec:squares}

In this section we prove Theorem~\ref{thm2} and Theorem~\ref{thm3} for quadrilaterals, i.e., in the case $N=4$. The proof exploits the same idea as in the triangle case, inspired by the arguments in \cite{CaHiVoYa2016}, and uses a continuous symmetrization to prove that the stationarity conditions corresponding to sliding and tilting first enforce the quadrilateral to be equilateral; and then, via a reflection argument, they imply that the polygon is also equiangular.

\begin{proof}[Proof of Theorem~\ref{thm2} in the case $N=4$]
Let $\p\in\pol_4$ be an arbitrary quadrilateral satisfying the conditions \eqref{eq:sliding} and \eqref{eq:tilting} such that the diagonal between $P_1$ and $P_3$ lies on the $x_1$-axis with $P_1$ on the negative $x_1$-axis, the midpoint of this diagonal coincides with the origin, and the vertex $P_2$ is in the upper half-plane. If $\p$ is not convex, we assume that the diagonal $\side{P_1}{P_3}$ is in the interior of the polygon. As in the case of a triangle, for any $x_2 \in \R$, we let $\p_{x_2}=\{x_1 \in \R \colon (x_1,x_2)\in\p\}\subset\R$. Then $\p_{x_2}=(c_{x_2}-r_{x_2},c_{x_2}+r_{x_2})$ for some $c_{x_2}\in\R$ and $r_{x_2}\geq 0$ denoting the center and the radius of the slice $\p_{x_2}$, respectively. Also, we define $d_2 \defeq \dist \big(P_2,\{x_1=0\}\big)$ and $d_4 \defeq \dist \big(P_4,\{x_1=0\}\big)$.

We assume by contradiction that $\alpha_2^- > \alpha_2^+$ where, as in the statement of Proposition~\ref{prop:ricIvar}, $\alpha_2^-$ and $\alpha_2^+$ are the angles between $\side{P_1}{P_2}$ and $\side{P_1}{P_3}$, and between $\side{P_1}{P_3}$ and $\side{P_2}{P_3}$, respectively (see Figure~\ref{fig:quad_carrillo}).
Let $\{\Phi_t\}_t$ be the flow defined by
\begin{equation}\label{eq:quad_flow}
\Phi_t(x_1,x_2) \defeq
\begin{cases}
\displaystyle (x_1+\beta^+ x_2\,t,x_2) & \text{if }x_2>0, \\[8pt]
\displaystyle (x_1- \beta^- x_2\,t,x_2) & \text{if }x_2<0,
\end{cases}
\end{equation}
where the constants $\beta^+,\beta^-\geq0$ are to be chosen later, so that
\begin{equation*}
\Phi_t(\p_{x_2}) =
\begin{cases}
\displaystyle \p_{x_2} + \beta^+ x_2\,t & \text{if }x_2>0, \\[8pt]
\displaystyle \p_{x_2} - \beta^- x_2\,t    & \text{if }x_2<0,
\end{cases}
\end{equation*}
and $\Phi_t(\p) = \big\{ (x_1,x_2) \colon x_1 \in \Phi_t(\p_{x_2}) \big\}$. Notice that the flow $\Phi_t$ is a superposition of the variations considered as in Definition~\ref{def:ricIvar} for $i=2,4$, with the two vertices moving with different velocities. One can check that
\begin{equation} \label{proof:quadrilaterals}
\frac{\dd}{\dd t} \Big|_{t=0} \nl_\delta(\Phi_t(\p)) = (\beta^+\ell_1\sin\alpha_2^-) I_2 - (\beta^-\ell_3\sin\alpha_4^-) I_4
\end{equation}
where $I_2$ and $I_4$ are the first variations computed in Proposition~\ref{prop:ricIvar} (here $\alpha_4^-$ is the angle between $\side{P_3}{P_4}$ and $\side{P_1}{P_3}$).

Again, we regularize the kernel by introducing a small parameter $\delta>0$ and take $K_\delta$ and $\nl_\delta$ as in \eqref{eq:Kdelta}. Set $K_{\delta,l}(r) \defeq K_\delta(\sqrt{l^2+r^2}) = K(\sqrt{l^2+r^2}+\delta)$ and note that it satisfies the assumptions of Lemma~\ref{lem:carrillo}. By Fubini's theorem, we write
	\[
		\nl_\delta(\Phi_t(\p)) = \int_{\R^2} \! \int_{\R^2} K_\delta(|x-y|)\chi_{\Phi_t(\p)}(x)\chi_{\Phi_t(\p)}(y)\dd x \dd y = \int_{\R}\!\int_{\R} \calI_{K_{\delta,l}}\big[\p_{x_2},\p_{y_2}\big](t) \dd x_2 \dd y_2,
	\]
where
	\begin{align*}
		\calI_{K_{\delta,l}}\big[\p_{x_2},\p_{y_2}\big](t) &\defeq \int_\R \! \int_\R K_{\delta,l}(x_1-y_1)\chi_{\Phi_t(\p_{x_2})}(x_1)\chi_{\Phi_t(\p_{y_2})}(y_1) \dd x_1 \dd y_1 \\
																						 &= \int_{-r_{x_2}}^{r_{x_2}}\!\int_{-r_{y_2}}^{r_{y_2}} K_{\delta,l}\big(x_1-y_1+c_{x_2}-c_{y_2}+\xi(x_2)t - \xi(y_2)t \big) \dd y_1 \dd x_1
	\end{align*}
with $l=|x_2-y_2|$ and
	\begin{equation} \label{eq:velocity}
	\xi(s)\defeq\begin{cases}
\beta^+ s 	& \text{if }s>0, \\
-\beta^- s   & \text{if }s<0.
\end{cases}
	\end{equation}

Now, differentiating the energy yields
	\begin{align*}
		\frac{\dd}{\dd t} \Big|_{t=0} \nl_\delta(\Phi_t(\p))  &= \int_{\R} \! \int_{\R} \frac{\dd}{\dd t}\Big|_{t=0} \calI_{K_{\delta,l}}\big[\p_{x_2},\p_{y_2}\big](t) \dd x_2 \dd y_2 \\
																							&=\int_0^{d_2}\!\int_0^{d_2} \frac{\dd}{\dd t}\Big|_{t=0} \calI_{K_{\delta,l}}\big[\p_{x_2},\p_{y_2}\big](t) \dd x_2 \dd y_2 \\
																							&\qquad + 2\int_0^{d_2}\!\int_{-d_4}^0 \frac{\dd}{\dd t}\Big|_{t=0} \calI_{K_{\delta,l}}\big[\p_{x_2},\p_{y_2}\big](t) \dd x_2 \dd y_2 \\
																							&\qquad\qquad + \int_{-d_4}^0\!\int_{-d_4}^0 \frac{\dd}{\dd t}\Big|_{t=0} \calI_{K_{\delta,l}}\big[\p_{x_2},\p_{y_2}\big](t) \dd x_2 \dd y_2.
	\end{align*}
By using Lemma~\ref{lem:carrillo} to estimate the first integral (since $\beta^+\geq0$) and the identity
	\[
		\frac{\dd}{\dd t}\Big|_{t=0} \calI_{K_{\delta,l}}\big[\p_{x_2},\p_{y_2}\big](t) = \big(\xi(x_2)-\xi(y_2)\big)\int_{-r_{x_2}}^{r_{x_2}}\!\int_{-r_{y_2}}^{r_{y_2}} K_{\delta,l}^\pr\big((x_1-y_1)+(c_{x_2}-c_{y_2})\big)\dd y_1 \dd x_1
	\]
to rewrite the second integral, we have that
\begin{align}
\frac{\dd}{\dd t} &\Big|_{t=0} \nl_\delta(\Phi_t(\p)) \notag\\
& \geq \beta^+ \int_0^{d_2}\!\int_0^{d_2} C_{K_{\delta,l}} \min\{r_{x_2},r_{y_2}\} |c_{x_2}-c_{y_2}|\, |x_2-y_2|\dd x_2 \dd y_2  \notag \\
&\quad\quad +2\int_0^{d_2}\!\int_{-d_4}^0 (\beta^+ x_2 + \beta^- y_2) \int_{-r_{x_2}}^{r_{x_2}}\!\int_{-r_{y_2}}^{r_{y_2}} K_{\delta,l}^\pr\big((x_1-y_1)+(c_{x_2}-c_{y_2})\big)\dd y_1 \dd x_1 \dd y_2 \dd x_2 \notag\\
&\quad\quad\quad +\int_{-d_4}^0\!\int_{-d_4}^0 \frac{\dd}{\dd t}\Big|_{t=0} \calI_{K_{\delta,l}}\big[\p_{x_2},\p_{y_2}\big](t) \dd x_2 \dd y_2 \label{eq:first_var_lb}
\end{align}
where $C_{K_{\delta,l}}$ is given by \eqref{eq:C_W}.

Now, let for $x_2\in[0,d_2]$ and $y_2\in[-d_4,0]$
	\[
		\calI_\delta(x_2,y_2) \defeq (\beta^+ x_2 + \beta^- y_2) \int_{-r_{x_2}}^{r_{x_2}}\!\int_{-r_{y_2}}^{r_{y_2}} K_{\delta,l}^\pr\big((x_1-y_1)+(c_{x_2}-c_{y_2})\big)\dd y_1 \dd x_1.
	\]
We will show that $\calI_\delta(x_2,y_2)\geq0$ for some $\beta^+,\, \beta^- \geq 0$. In order to achieve this estimate we distinguish between two cases: (i) $P_4$ lies in the fourth quadrant of the $(x_1,x_2)$-plane, or (ii) $P_4$ lies in the third quadrant of the $(x_1,x_2)$-plane (see Figure \ref{fig:quad_carrillo}).

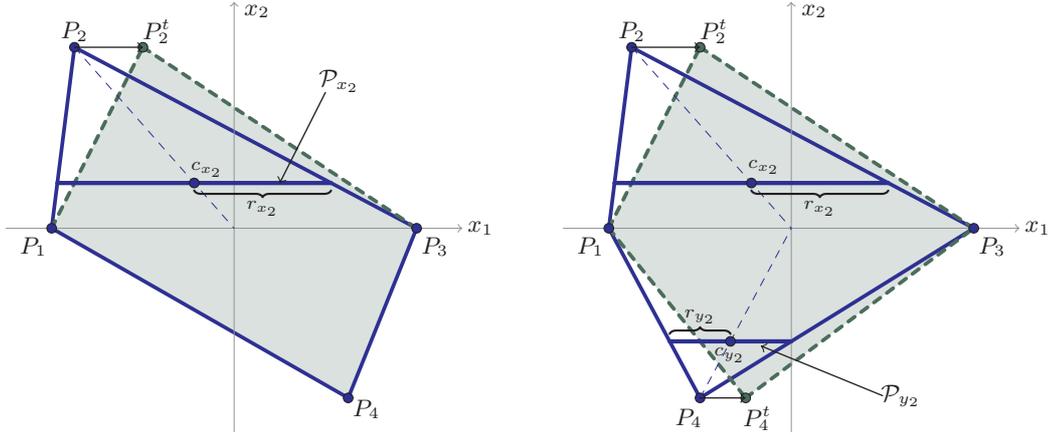
\begin{figure}[ht]
\definecolor{qqqqff}{rgb}{0,0,1}
\definecolor{qqwuqq}{rgb}{0,0.4,0}
\begin{tikzpicture}[scale=0.6,line cap=round,line join=round]
\clip(-6,-4.5) rectangle (6,5);
\fill[line width=0pt,fill=qqwuqq,fill opacity=0.2] (-4,0) -- (-2,4) -- (4,0) -- (2.5,-3.75) -- cycle;
\draw [line width=1.4pt,color=qqqqff] (-4,0)-- (-3.5,4);
\draw [line width=1.4pt,color=qqqqff] (-3.5,4)-- (4,0);
\draw [line width=1.4pt,color=qqqqff] (4,0)-- (2.5,-3.75);
\draw [line width=1.4pt,color=qqqqff] (2.5,-3.75)-- (-4,0);
\draw [line width=1.4pt,dash pattern=on 3pt off 3pt,color=qqwuqq] (-4,0)-- (-2,4);
\draw [line width=1.4pt,dash pattern=on 3pt off 3pt,color=qqwuqq] (-2,4)-- (4,0);
\draw [->,line width=0.3pt,color=black!50] (-5,0)-- (5,0);
\draw [->,line width=0.3pt,color=black!50] (0,-5)-- (0,5);
\draw [line width=0.3pt,dash pattern=on 3pt off 3pt,color=qqqqff] (-3.5,4) -- (0,0);
\draw [line width=1.6pt,color=qqqqff] (-31/8,1) -- (17/8,1);
\begin{footnotesize}
\draw (5.4,0) node {$x_1$};
\draw (0.5,4.8) node {$x_2$};
\draw [fill=qqqqff] (-4,0) circle (3pt);
\draw (-4.4,-0.4) node {$P_1$};
\draw [fill=qqqqff] (4,0) circle (3pt);
\draw (4.4,-0.4) node {$P_3$};
\draw [fill=qqqqff] (-3.5,4) circle (3pt);
\draw (-3.5,4.35) node {$P_2$};
\draw [fill=qqqqff] (2.5,-3.75) circle (3pt);
\draw (2.9,-4) node {$P_4$};
\draw [fill=qqqqff] (-7/8,1) circle (3pt);
\draw (-0.6,1.3) node {{\tiny $c_{x_2}$}};
\draw [fill=qqwuqq] (-2,4) circle (3pt);
\draw (-1.7,4.35) node {$P_2^t$};
\draw (2.3,3.25) node {$\p_{x_2}$};
\draw [
    thick,
    decoration={
        brace,
        mirror,
        raise=0.1cm
    },
    decorate
] (-7/8,1) -- (17/8,1) 
node [pos=0.5,anchor=north,yshift=-0.12cm] {{\tiny $r_{x_2}$}}; 
\draw [->,line width=0.5pt] (2,3) -- (1,1);
\end{footnotesize}
\draw [->,line width=0.5pt] (-3.5,4)-- (-2,4);
\end{tikzpicture}
\begin{tikzpicture}[scale=0.6,line cap=round,line join=round]
\clip(-6,-4.5) rectangle (6,5);
\fill[line width=0pt,fill=qqwuqq,fill opacity=0.2] (-4,0) -- (-2,4) -- (4,0) -- (-1,-3.75) -- cycle;
\draw [line width=1.4pt,color=qqqqff] (-4,0)-- (-3.5,4);
\draw [line width=1.4pt,color=qqqqff] (-3.5,4)-- (4,0);
\draw [line width=1.4pt,color=qqqqff] (4,0)-- (-2,-3.75);
\draw [line width=1.4pt,color=qqqqff] (-2,-3.75)-- (-4,0);
\draw [line width=1.4pt,dash pattern=on 3pt off 3pt,color=qqwuqq] (-4,0)-- (-2,4);
\draw [line width=1.4pt,dash pattern=on 3pt off 3pt,color=qqwuqq] (-2,4)-- (4,0);
\draw [line width=1.4pt,dash pattern=on 3pt off 3pt,color=qqwuqq] (-4,0)-- (-1,-3.75);
\draw [line width=1.4pt,dash pattern=on 3pt off 3pt,color=qqwuqq] (-1,-3.75)-- (4,0);
\draw [->,line width=0.3pt,color=black!50] (-5,0)-- (5,0);
\draw [->,line width=0.3pt,color=black!50] (0,-5)-- (0,5);
\draw [line width=0.3pt,dash pattern=on 3pt off 3pt,color=qqqqff] (-3.5,4) -- (0,0);
\draw [line width=0.3pt,dash pattern=on 3pt off 3pt,color=qqqqff] (-2,-3.75) -- (0,0);
\draw [line width=1.6pt,color=qqqqff] (-31/8,1) -- (17/8,1);
\draw [line width=1.6pt,color=qqqqff] (-8/3,-2.5) -- (0,-2.5);
\begin{footnotesize}
\draw (5.4,0) node {$x_1$};
\draw (0.5,4.8) node {$x_2$};
\draw [fill=qqqqff] (-4,0) circle (3pt);
\draw (-4.4,-0.4) node {$P_1$};
\draw [fill=qqqqff] (4,0) circle (3pt);
\draw (4.4,-0.4) node {$P_3$};
\draw [fill=qqqqff] (-3.5,4) circle (3pt);
\draw (-3.5,4.35) node {$P_2$};
\draw [fill=qqqqff] (-2,-3.75) circle (3pt);
\draw (-2.25,-4.2) node {$P_4$};
\draw [fill=qqwuqq] (-1,-3.75) circle (3pt);
\draw (-0.75,-4.2) node {$P_4^t$};
\draw [fill=qqqqff] (-7/8,1) circle (3pt);
\draw (-0.6,1.35) node {{\tiny $c_{x_2}$}};
\draw [fill=qqqqff] (-8/6,-2.5) circle (3pt);
\draw (-1.35,-2.8) node {{\tiny $c_{y_2}$}};
\draw [fill=qqwuqq] (-2,4) circle (3pt);
\draw (-1.7,4.35) node {$P_2^t$};
\draw (2.4,-3.7) node {$\p_{y_2}$};
\draw [
    thick,
    decoration={
        brace,
        mirror,
        raise=0.1cm
    },
    decorate
] (-7/8,1) -- (17/8,1) 
node [pos=0.5,anchor=north,yshift=-0.12cm] {{\tiny $r_{x_2}$}}; 
\draw [
    thick,
    decoration={
        brace,
        raise=0.1cm
    },
    decorate
] (-8/3,-2.5) -- (-8/6,-2.5) 
node [pos=0.5,anchor=north,yshift=0.55cm] {{\tiny $r_{y_2}$}}; 
\draw [->,line width=0.5pt] (2,-3.7) -- (-4/6,-2.6);
\end{footnotesize}
\draw [->,line width=0.5pt] (-3.5,4)-- (-2,4);
\draw [->,line width=0.5pt] (-2,-3.75)-- (-1,-3.75);
\end{tikzpicture}
\caption{The variation considered in the proof of Theorem~\ref{thm2}, $N=4$: Case~(i) (left) and Case~(ii) (right).}\label{fig:quad_carrillo}
\end{figure}

\medskip\noindent \emph{Case (i).} ($P_4$ lies in the fourth quadrant of the $(x_1,x_2)$-plane.) Since $\alpha_2^- > \alpha_2^+$ by assumption, for $x_2>0$ the center of the slice $\p_{x_2}$ is given by $c_{x_2}=\zeta x_2$ where $\zeta<0$ is the slope of the line passing through the origin and the vertex $P_2$. We choose $\beta^+ = -\zeta >0$ and $\beta^-=0$ in \eqref{eq:quad_flow}. Note that, since $P_2$ and $P_4$ are on the opposite sides of the $x_2$-axis, we have that $c_{x_2}-c_{y_2}<0$.

As in the proof of Lemma~\ref{lem:carrillo}, we have
	\[
		\calI_\delta(x_2,y_2) = \beta^+ x_2 \iint_R K^\pr_{\delta,l}(x_1-y_1) \dd x_1 \dd y_1,
	\]
where
	\begin{equation} \label{eq:carrillo_rectangle1}
		R \defeq \big[-r_{x_2}+c_{x_2}-c_{y_2},r_{x_2}+c_{x_2}-c_{y_2}\big] \times \big[ -r_{y_2},r_{y_2} \big].
	\end{equation}
Since $R$ is a rectangle centered at $(\frac{c_{x_2}-c_{y_2}}{2},0)$ and $c_{x_2}-c_{y_2}<0$, we have that for every $h>0$
\begin{equation} \label{eq:slices}
\mathcal{L}^1\bigl( R\cap\{y_1=x_1+h\}\bigr) \geq \mathcal{L}^1\bigl(R\cap\{y_1=x_1-h\}\bigr).
\end{equation}
Therefore
\begin{align*}
\calI_\delta(x_2,y_2)
& = \frac{\beta^+ x_2}{\sqrt{2}} \int_0^{+\infty} \biggl( \int_{R\cap\{y_1=x_1+h\}}K^\pr_{\delta,l}(-h)\dd\mathcal{L}^1 + \int_{R\cap\{y_1=x_1-h\}}K^\pr_{\delta,l}(h)\dd\mathcal{L}^1 \biggr) \dd h \\
& = \frac{\beta^+ x_2}{\sqrt{2}} \int_0^{+\infty} K^\pr_{\delta,l}(h) \Bigl( \mathcal{L}^1(R\cap\{y_1=x_1-h\}) - \mathcal{L}^1(R\cap\{y_1=x_1+h\}) \Bigr) \dd h \\
& \geq0,
\end{align*}
by \eqref{eq:slices} and the fact that $K^\pr_{\delta,l}(h)<0$ for $h>0$.

\medskip\noindent \emph{Case (ii).} ($P_4$ lies in the third quadrant of the $(x_1,x_2)$-plane.) Now, given any $x_2\in \R$, the center of the slice $\p_{x_2}$ is given by $c_{x_2}=\zeta x_2$ if $x_2>0$ and by $c_{x_2}=\eta x_2$ if $x_2<0$, where $\zeta<0$ and $\eta>0$ are two constants given by the slopes of the lines passing through the origin and the vertices $P_2$ and $P_4$, respectively. In this case we choose $\beta^+ = -\zeta >0$ and $\beta^-=\eta>0$ in \eqref{eq:quad_flow}, and, as before, rewrite $\calI_\delta(x_2,y_2)$ as 
	\[
		\calI_\delta(x_2,y_2) = (\beta^+ x_2 + \beta^- y_2) \iint_R K^\pr_{\delta,l}(x_1-y_1) \dd x_1 \dd y_1,
	\]
where $R$ is the rectangle defined by \eqref{eq:carrillo_rectangle1}.

Suppose $\beta^+ x_2 + \beta^- y_2>0$. Then
$$
c_{x_2}-c_{y_2}=\zeta x_2 - \eta y_2 = -(\beta^+ x_2+\beta^- y_2)<0,
$$
and, as in the previous case, $R$ is a rectangle centered on the negative $x_1$-axis, hence we get that $\calI_\delta(x_2,y_2)\geq0$.

Suppose $\beta^+ x_2 + \beta^- y_2 <0$. Then $c_{x_2}-c_{y_2}>0$, and therefore the center of the rectangle $R$ is on the positive $x_1$-axis: it follows that for every $h>0$
\begin{equation*}
\mathcal{L}^1\bigl( R\cap\{y_1=x_1+h\}\bigr) \leq \mathcal{L}^1\bigl(R\cap\{y_1=x_1-h\}\bigr).
\end{equation*}
Hence we get that
\begin{equation*}
\int_0^{+\infty} K^\pr_{\delta,l}(h) \Bigl( \mathcal{L}^1(R\cap\{y_1=x_1-h\}) - \mathcal{L}^1(R\cap\{y_1=x_1+h\}) \Bigr) \dd h \leq 0,
\end{equation*}
and since $\beta^+ x_2 + \beta^- y_2<0$, we conclude that $\calI_\delta(x_2,y_2)\geq0$.

\medskip\noindent\emph{Conclusion.} We proved that in both cases, for a suitable choice of $\beta^+>0$ and $\beta^- \geq 0$, we have $\calI_\delta(x_2,y_2) \geq 0$ for every $x_2\in[0,d_2]$ and $y_2\in[-d_4,0]$.

Going back to \eqref{eq:first_var_lb} we obtain that
\begin{align*}
\frac{\dd}{\dd t} \Big|_{t=0} \nl_\delta(\Phi_t(\p))
& \geq \beta^+ \int_0^{d_2}\!\int_0^{d_2} C_{K_{\delta,l}} \min\{r_{x_2},r_{y_2}\} |c_{x_2}-c_{y_2}|\, |x_2-y_2|\dd x_2 \dd y_2 \\
&\quad\quad\quad +\int_{-d_4}^0\!\int_{-d_4}^0 \frac{\dd}{\dd t}\Big|_{t=0} \calI_{K_{\delta,l}}\big[\p_{x_2},\p_{y_2}\big](t) \dd x_2 \dd y_2 .
\end{align*}
Concerning the second integral above, it is sufficient to observe that it is equal to zero in Case~(i) (since $\beta^-=0$), and it is nonnegative in Case~(ii) as a consequence of Lemma~\ref{lem:carrillo} (since $\beta^-\geq0$). Therefore, recalling the definition \eqref{eq:C_W} of $C_{K_{\delta,l}}$, we have
	\[
		\frac{\dd}{\dd t} \Big|_{t=0} \nl_\delta(\Phi_t(\p)) \geq C_\delta > 0,
	\]
for a constant $C_\delta$ bounded away from zero uniformly in $\delta$. Again, as in Proposition~\ref{prop:firstvar}, we can pass to the limit $\delta\to 0$ and get that $\frac{\dd}{\dd t} \big|_{t=0} \nl(\Phi_t(\p)) > 0$. However, this contradicts \eqref{proof:quadrilaterals}, since in view of Proposition~\ref{prop:ricIvar2} we have $I_2=I_4=0$. This proves that $\alpha_2^-=\alpha_2^+$.

By swapping $P_2$ and $P_4$ in the arguments above yields that, in fact, $\theta_1=\theta_3$. Now, repeating the same arguments for the vertices $P_1$ and $P_3$ (that is, taking the diagonal $\side{P_2}{P_4}$ as the direction of symmetrization), we obtain that $\theta_2=\theta_4$, i.e., that $\p$ is a rhombus.

\begin{figure}[ht]
\definecolor{qqqqff}{rgb}{0,0,1}
\definecolor{qqwuqq}{rgb}{0,0.4,0}
\begin{tikzpicture}[scale=0.3,line cap=round,line join=round]
\clip(-12.5,-1) rectangle (12.5,13.5);
\fill[line width=0pt,fill=qqwuqq,fill opacity=0.2] (-6.5,0) -- (-11.5,12) -- (-1.5,12) -- cycle;
\fill[line width=0pt,fill=qqqqff,fill opacity=0.2] (6.5,0) -- (1.5,12) -- (11.5,12) -- cycle;
\draw [line width=1.4pt,color=qqqqff] (-6.5,0)-- (-1.5,12) -- (11.5,12) -- (6.5,0) -- (-6.5,0);
\draw [line width=1.4pt,dash pattern=on 3pt off 3pt,color=qqwuqq] (-6.5,0)-- (-11.5,12) --(1.5,12) -- (6.5,0);
\draw [->,line width=0.3pt,color=black!50] (-10,0)-- (10,0);
\draw [->,line width=0.3pt,color=black!50] (0,-0.5)-- (0,13.5);
\begin{tiny}
\draw (10.5,-0.2) node {$x_1$};
\draw (0.7,13.2) node {$x_2$};
\draw [fill=qqqqff] (-6.5,0) circle (3pt);
\draw (-6.4,-0.6) node {$P_1$};
\draw [fill=qqqqff] (6.5,0) circle (3pt);
\draw (6.4,-0.6) node {$P_4$};
\draw [fill=qqqqff] (-1.5,12) circle (3pt);
\draw (-2,12.6) node {$P_2$};
\draw [fill=qqqqff] (11.5,12) circle (3pt);
\draw (12,12.6) node {$P_3$};
\draw [fill=qqwuqq] (-11.5,12) circle (3pt);
\draw (-12,12.6) node {$\tld{P}_4$};
\draw [fill=qqwuqq] (1.5,12) circle (3pt);
\draw (2,12.6) node {$\tld{P}_2$};
\draw [fill=qqqqff] (0,0) circle (3pt);
\draw (0.7,-0.6) node {$M_4$};
\end{tiny}
\begin{footnotesize}
\draw [color=qqwuqq] (-6.4,6.5) node {$\tld{D}$};
\draw [color=qqqqff] (6.4,6.5) node {$D$};
\end{footnotesize}
\end{tikzpicture}
\caption{A reflection argument shows that if $\p$ is rhombus and satisfies \eqref{eq:tilting} then $\p$ has to be a square. Here the reflection of $\p$ in the $x_2$-axis is the rhombus $\tld{\p}$ depicted with the dashed lines.}\label{fig:rhombus}
\end{figure}
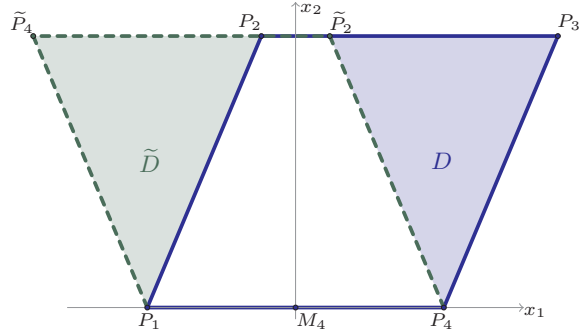

Finally, we are going to use a reflection argument similar to the one in the first proof of Theorem~\ref{thm2} in the case $N=3$ in order to conclude that $\p$ is a square. Since $\nl$ is invariant under rigid transformations, suppose the side $\side{P_1}{P_4}$ lies on the $x_1$-axis and the midpoint $M_4$ coincides with the origin. Suppose $\theta_1<\theta_4$. Let $\tld{\p}$ denote the reflection of $\p$ with respect to the $x_2$-axis, and define the sets $D \defeq \p \setminus \tld{\p}$ and $\tld{D} \defeq \tld{\p} \setminus \p$ (see Figure \ref{fig:rhombus}). Let $x\in \side{M_4}{P_4}$ and denote by $\tld{x} \in \side{P_1}{M_4}$ the reflection of $x$ in the $x_2$-axis. Then the same calculation as in \eqref{eq:reflection_calc} shows that	$v_{\tld{\p}}(x)  - v_{\p}(x)< 0$. Again, multiplying both sides by $|x-M_4|$ and integrating, then, yields a contradicts with the condition \eqref{eq:tilting}; hence, $\theta_1=\theta_4$, and we conclude that $\p$ is a square.
\end{proof}

\begin{proof}[Proof of Theorem~\ref{thm3} in the case $N=4$]
By repeating the proof of Theorem~\ref{thm2}, assuming by contradiction that $\alpha_2^->\alpha_2^+$ we obtain that the first variation \eqref{proof:quadrilaterals} is strictly positive, namely
\begin{equation} \label{proof:quadrilaterals2}
(\beta^+\ell_1\sin\alpha_2^-) I_2 - (\beta^-\ell_3\sin\alpha_4^-)I_4 > 0.
\end{equation}
We distinguish between Case~(i) and Case~(ii), as before.
In Case~(i) we had $\beta^-=0$ and $\beta^+>0$, therefore we deduce from \eqref{proof:quadrilaterals2} that $I_2>0$. However, this contradicts the conclusion of Proposition~\ref{prop:ricIvar2}, which yields $I_2 = 2\bar{\sigma}\bigl(\cos\alpha_2^- - \cos\alpha_2^+\bigr)<0$ since by assumption $\alpha_2^->\alpha_2^+$.
Case~(ii) corresponds to the assumption $\alpha_4^-<\alpha_4^+$. In this case, again by Proposition~\ref{prop:ricIvar2} we have
\begin{align*}
0 & \xupref{proof:quadrilaterals2}{<} (\beta^+\ell_1\sin\alpha_2^-) I_2 - (\beta^-\ell_3\sin\alpha_4^-)I_4 \\
& \xupref{stationarity-ric2}{=} 2\bar{\sigma}(\beta^+\ell_1\sin\alpha_2^-) (\cos\alpha_2^- - \cos\alpha_2^+) - 2\bar{\sigma}(\beta^-\ell_3\sin\alpha_4^-)(\cos\alpha_4^- - \cos\alpha_4^+) <0
\end{align*}
since $\alpha_2^->\alpha_2^+$ and $\alpha_4^-<\alpha_4^+$. Therefore $\alpha_2^-=\alpha_2^+$ and, by repeating the argument for the other pairs of sides, we obtain that $\p$ must be a rhombus.

The conclusion follows now by the same reflection argument as in the proof of Theorem~\ref{thm2}: assuming $\theta_1<\theta_4$, we obtain
$$
\int_{\side{P_4}{M_4}}v_\p(x)|x-M_4|\dd\Hone(x) - \int_{\side{P_1}{M_4}}v_\p(x)|x-M_4|\dd\Hone(x)>0.
$$
However, by \eqref{eq:tilting2} the previous quantity is equal to $\frac{\bar{\sigma}\ell_4}{2}\bigl(\psi(\theta_4)-\psi(\theta_1)\bigr)$, which is negative since $\theta_1<\theta_4$ and $\psi$ is decreasing. This contradiction proves that $\theta_1=\theta_4$ and therefore $\p$ is a square.
\end{proof}


\subsection*{Acknowledgments}
The authors would like to thank Ilaria Fragalà and Gian Paolo Leonardi for interesting comments. MB is member of 2020 INdAM - GNAMPA project \textit{Variational Analysis of nonlocal models in applied science}.

\appendix
\section{First variation of the nonlocal energy}\label{sec:Appendix}

\begin{proof}[Proof of Proposition \ref{prop:firstvar}]
The proof follows the same strategy used in \cite{BoCr14} to compute the first variation in the particular case of a Riesz kernel. We regularize the kernel by introducing a small parameter $\delta>0$ and by setting
\begin{equation} \label{eq:Kdelta}
K_\delta(r)\defeq K(r+\delta), \qquad\nl_\delta(E)\defeq \int_E\int_E K_\delta(|x-y|)\dd x \dd y,
\end{equation}
so that $K_\delta\in C^1([0,+\infty))$. 

We let $\Phi_t(x)\defeq\Phi(x,t)$ and $J\Phi_t(x)\defeq \det(D\Phi_t(x))$ denote the Jacobian of the map $\Phi_t$. Since the regularized kernel $K_\delta$ is of class $C^1$ and the flow $\Phi$ is of class $C^2$, the map $t \mapsto K_\delta(|\Phi_t(x)-\Phi_t(y)|)J\Phi_t(x)J\Phi_t(y)$ is $C^1$, and its derivative is uniformly bounded on the bounded set $E\times E$ for every small $t$. Therefore we can bring the derivative inside the integral and all the following computations are justified. By a change of variables we obtain for $t\in(-\bar{t},\bar{t})$
\begin{align*}
\frac{\dd}{\dd t}\nl_\delta& (\Phi_t(E))
= \frac{\dd}{\dd t}\int_E\int_E K_\delta(|\Phi_t(x)-\Phi_t(y)|)J\Phi_t(x)J\Phi_t(y)\dd x \dd y \\
& = 2\int_E\int_E K_\delta(|\Phi_t(x)-\Phi_t(y)|)J\Phi_t(x)\frac{\partial J\Phi_t}{\partial t}(y)\dd x\dd y \\
& \hskip.5cm + 2\int_E\int_E K'_\delta(|\Phi_t(x)-\Phi_t(y)|)\frac{\Phi_t(x)-\Phi_t(y)}{|\Phi_t(x)-\Phi_t(y)|}\cdot\frac{\partial\Phi(x,t)}{\partial t} \, J\Phi_t(x)J\Phi_t(y)\dd x \dd y \\
& = 2\int_E\int_E K_\delta(|\Phi_t(x)-\Phi_t(y)|)J\Phi_t(x)\frac{\partial J\Phi_t}{\partial t}(y)\dd x\dd y \\
& \hskip.5cm + 2\int_E\int_E \Bigl[ \nabla_x\bigl(K_\delta(|\Phi_t(x)-\Phi_t(y)|)\bigr)\bigl(D\Phi_t(x)\bigr)^{-1}\Bigr] \cdot\frac{\partial\Phi(x,t)}{\partial t} \, J\Phi_t(x)J\Phi_t(y)\dd x \dd y.
\end{align*}
Now integrating by parts in the last integral we obtain
\begin{align*}
\frac{\dd}{\dd t}\nl_\delta& (\Phi_t(E))
= 2\int_E\int_E K_\delta(|\Phi_t(x)-\Phi_t(y)|)J\Phi_t(x)\frac{\partial J\Phi_t}{\partial t}(y)\dd x\dd y \\
& - 2\int_E\int_E K_\delta(|\Phi_t(x)-\Phi_t(y)|) \dive_x\Bigl( \frac{\partial\Phi(x,t)}{\partial t}\bigl(D\Phi_t(x)\bigr)^{-T}J\Phi_t(x)J\Phi_t(y)\Bigr)\dd x \dd y \\
& + 2\int_E\int_{\partial E} K_\delta(|\Phi_t(x)-\Phi_t(y)|) \frac{\partial\Phi(x,t)}{\partial t}\bigl(D\Phi_t(x)\bigr)^{-T}\cdot\nu_E(x)J\Phi_t(x)J\Phi_t(y)\dd\Hone(x) \dd y ,
\end{align*}
that is 
\begin{equation*}
\begin{split}
\frac{\dd}{\dd t}\nl_\delta (\Phi_t(E))
& = \int_E\int_E K_\delta(|\Phi_t(x)-\Phi_t(y)|)h_1(x,y,t)\dd x \dd y \\
& \qquad + \int_E\int_{\partial E}K_\delta(|\Phi_t(x)-\Phi_t(y)|)h_2(x,y,t)\dd\Hone(x)\dd y,
\end{split}
\end{equation*}
where we set
\begin{equation*}
h_1(x,y,t)\defeq 2J\Phi_t(x)\frac{\partial J\Phi_t}{\partial t}(y) - 2 \dive_x\Bigl( \frac{\partial\Phi(x,t)}{\partial t}\bigl(D\Phi_t(x)\bigr)^{-T}J\Phi_t(x)J\Phi_t(y)\Bigr),
\end{equation*}
\begin{equation*}
h_2(x,y,t)\defeq 2\frac{\partial\Phi(x,t)}{\partial t}\bigl(D\Phi_t(x)\bigr)^{-T}\cdot\nu_E(x)J\Phi_t(x)J\Phi_t(y).
\end{equation*}
By using the definition \eqref{eq:Kdelta} of $K_\delta$ and the fact that the functions $h_1(x,y,t)$ and $h_2(x,y,t)$ are uniformly bounded, one can then show that $\nl_\delta(\Phi_t(E))\to\nl(\Phi_t(E))$ and
\begin{equation*}
\begin{split}
\frac{\dd}{\dd t}\nl_\delta (\Phi_t(E)) \to H(t) & \defeq \int_E\int_E K(|\Phi_t(x)-\Phi_t(y)|)h_1(x,y,t)\dd x \dd y \\
& \qquad + \int_E\int_{\partial E}K(|\Phi_t(x)-\Phi_t(y)|)h_2(x,y,t)\dd\Hone(x)\dd y,
\end{split}
\end{equation*}
as $\delta\to0$, uniformly with respect to $t\in[-\bar{t},\bar{t}]$. Therefore we conclude that
\begin{equation*}
\frac{\dd}{\dd t}\Big|_{t=0} \nl (\Phi_t(E)) = H(0) = 2\int_E\int_{\partial E}K(|x-y|) X(x)\cdot\nu_E(x)\dd \Hone(x)\dd y,
\end{equation*}
where we used the Taylor expansion $\Phi_t(x)=x+tX(x)+o(t)$, from which it follows, in particular, the identity $\frac{\partial J\Phi_t}{\partial t}|_{t=0}=\dive X$.
\end{proof}

\bibliographystyle{IEEEtranS}
\def\url#1{}
\bibliography{references}

\end{document}